\documentclass[12pt,a4paper]{amsart}
\usepackage[english]{babel}
\usepackage{latexsym}
\usepackage{amssymb}
\usepackage{amscd}

\usepackage[all]{xy}
\usepackage[cp850]{inputenc}
\usepackage[mathscr]{eucal}
\theoremstyle{plain}

\newtheorem{teor}{Theorem}[section]
\newtheorem{proposition}[teor]{Proposition}
\newtheorem{cor}[teor]{Corollary}
\newtheorem{lema}[teor]{Lemma}
\theoremstyle{remark}

\newcommand{\R}{\mathbb{R}}
\newcommand{\N}{\mathbb{N}}

\def\e{\varepsilon}

\newcommand{\adef}{\begin{defn}}
\newcommand{\zdef}{\end{defn}}
\newtheorem{defn}[teor]{Definition}

\def\dim{\operatorname{dim}}

\def\PO{\operatorname{PO}}

\newcommand{\aproof}{\begin{proof}}
\newcommand{\zproof}{\end{proof}}

\newcommand{\U}{\mathscr U}

\begin{document}

\title{Banach spaces of almost universal\\ complemented disposition}

\author{Jes\'{u}s M .F. Castillo}
\address{Departamento de Matem\'aticas\\ Universidad de Extremadura\\
Avenida de Elvas\\ 06071-Badajoz\\ Spain} \email{castillo@unex.es}
\author{Yolanda Moreno}
\address{Departamento de Matem\'aticas\\ Universidad de Extremadura\\
Avenida de Elvas\\ 06071-Badajoz\\ Spain} \email{ymoreno@unex.es}

\thanks{This research has been supported in part by project MTM2016-76958-C2-1-P and Project IB16056 de la Junta de Extremadura}
\thanks{\emph{Acknowledgement.} The authors need to thank the truly remarkable job of the referee.}

\subjclass[2010]{46B03, 46M40}
\maketitle

\begin{abstract}  We introduce and study the notion of space of almost universal complemented disposition (a.u.c.d.) as a generalization of Kadec space. We show that every Banach space with separable dual is isometrically contained as a $1$-complemented subspace of a separable a.u.c.d. space and that all a.u.c.d. spaces with $1$-FDD are isometric and contain isometric $1$-complemented copies of every separable Banach space with $1$-FDD. We then study spaces of universal complemented disposition (u.c.d.) and provide different constructions for such spaces. We also consider spaces of u.c.d. with respect to separable spaces.\end{abstract}

\section{Introduction}
This paper can be considered a study of the properties that make the Kadec space \cite{kade} as it is: a separable Banach space containing complemented copies of every separable Banach space with the Bounded Approximation Property. This universal property, which apparently is a global property, is shown here to actually be a local property, called almost universal complemented disposition (a.u.c.d.) and shown to be very similar to Gurariy's almost universal disposition property \cite{gurari}. To emphasize the knot between the Kadec and Gurariy spaces, let us briefly survey the history of (complementably) universal Banach spaces and spaces of universal (complemented) disposition.\medskip

The topic of Banach spaces of universal and almost universal disposition with respect to a class $\mathcal M$ has its inception in the paper \cite{gurari} of Gurariy, who constructed a separable space $\mathcal G$ with the property that every into isometry from a finite dimensional space $F$ into $\mathcal G$ can be extended to an into almost isometry $G\to \mathcal G$ on every finite dimensional superspace $G$ of $F$. Several papers established the isometric uniqueness of $\mathcal G$ \cite{lusky}, its maximality \cite{p-w} and gave different descriptions for $\mathcal G$ \cite{luskysurv,lazarlind,woj}. Gurariy also conjectured the existence of Banach spaces of universal disposition (without ``almost") and of spaces of universal disposition with respect to the class $\mathcal S$ of separable spaces. This conjecture was proved to be true in  \cite{accgm2}, where a general method to construct spaces of universal disposition with respect to different classes $\mathcal M$ was presented. In particular, it was shown that the space that Gurariy conjectured is isometric to the Fra\"iss\'e limit, in the category of separable Banach spaces and into isometries, constructed by  Kubis \cite{kubis}. More recently, the papers \cite{cgk} (resp. \cite{bagekuba}) extend the method of \cite{accgm2} (resp. \cite{kubis}) to the study of quasi-Banach (resp. Fr\'echet) spaces of (almost) universal
disposition.\medskip

The topic of universal Banach spaces for a given class $\mathcal M$, i.e., Banach spaces in $\mathcal M$ containing an isometric/isomorphic copy of every space in $\mathcal M$ is another classical one (see the monograph of Dodos \cite{dodo} to find updated information). The two germinal  results are: the well known fact that $C[0,1]$ is isometrically universal for the class of separable spaces and Pe\l czy\'nsky's construction \cite{pelcuni} of a space $\mathcal P$ isomorphically universal for the class of Banach spaces with basis. Both results are relevant for the study in this paper. For instance, as the authors of \cite{cgk} remark, the fact that a separable space of almost universal disposition is also separably universal (cf. Corollary \ref{thesemiedge} in this paper) depends, in principle, on the isometric uniqueness of the Gurariy space.\medskip

The topic of complementably universal space for a class $\mathcal M$, i.e., spaces in $\mathcal M$ containing
complemented copies of every space in $\mathcal M$, also contains many interesting results, sometimes requiring descriptive set theory techniques; see, for instance, \cite[Theorem 1.2]{kurka}. The topic emerges in 1969 when Pe\l czy\'nski \cite{pelcuni} shows that the space $\mathcal P$ above mentioned is complementably universal for the class of Banach spaces with basis. In 1971 Kadec \cite{kade} obtains a complementably universal member $\mathcal K$ for the class of separable Banach spaces with the Bounded Approximation Property (BAP); still in 1971 Pe{\l}czy\'nski and Wojtaszczyk \cite{p-w} prove that also the class of separably spaces with Finite Dimensional Decompositions has a complementably universal member $\mathcal {PW}$. The classical results of Pe\l czy\'nski \cite{pelcbap} (resp. Pe{\l}czy\'nski-Wojtaszczyk) asserting that a separable Banach space  has the BAP if and only if it is complemented in a space with basis (resp. FDD) implies that the spaces $\mathcal P, \mathcal K$ and $\mathcal{PW}$ contain complemented copies of all separable spaces with BAP. Kalton \cite{kaltuni} performs a study of universal and complementably universal $F$-spaces, and remarks ``there are a number of other existence and non-existence results
known for other classes of separable spaces". It cannot go however without saying that Johnson and Szankowski \cite{johnszank1} showed that no separable complementably universal space exists for the class of separable Banach spaces. A related topic is that of when a Banach space with a property $P$ can be embedded into some Banach space with a finite dimensional decomposition \emph{and} property $P$. See, e.g., \cite{kurka,os}.

All spaces $\mathcal K$, $\mathcal P$ and $\mathcal{PW}$ are isomorphic (see Lemma \ref{mix}). The isometric uniqueness of complementably universal spaces is a different thing. Garbulinska recovered in \cite{garbula} the Fra\"iss\'e limit approach to construct first
a complementably universal space $\mathscr G$ for the class of separable spaces with FDD (thus isomorphic to the spaces of Kadec and Pe\l
czy\'nski) with a certain local isometric property. A closer inspection of the property that makes the space $\mathscr  G$ isometrically unique reveals that it is a local property, the one we have called almost universal complemented disposition and which is the object of study in this paper.\medskip

Sometimes the notation and results we present are rather technical. Thus, to encourage the reader and ease his way to the precise statements and proofs, let us briefly present intuitive versions of the main results in this paper. First of all, the key definition: A Banach space $E$ will be called of \emph{almost universal complemented disposition} (a.u.c.d. in short) if whenever one has an isometric embedding $F \to G$ between finite dimensional spaces with complemented range, every isometric embedding $F\to E$ with complemented range admits an extension to an almost isometric embedding $G\to E$ with complemented range. Regarding separable a.u.c.d. spaces our main results are:
\begin{itemize}
\item There exist separable spaces of almost universal complemented disposition. In fact, every Banach space with separable dual can be isometrically embedded as a $1$-complemented subspace of a separable a.u.c.d. space (Theorem \ref{KX}).
\item The separable spaces of a.u.c.d. are not unique. However, all separable spaces of a.u.c.d. with a $1$-FDD are isometric (Proposition \ref{noisomo} and Theorem \ref{uniqGarbu}).
\item Every separable space of a.u.c.d. with a 1-FDD contains isometric copies of every separable space and isometric 1-complemented copies of every separable space with a $1$-FDD (Theorem \ref{garbfdd}).
\end{itemize}

The pi{\`{e}}ce of resistance of our analysis is the so-called Approximation Lemma \ref{compfelix} that roughly says that if $E$ is Banach space of almost universal complemented disposition with a $1$-FDD then every almost isometry $F\to E$ from a finite dimensional space $F$ having almost complemented range can be approximated by an into isometry with complemented range.

Thus, since the Gurariy space $\mathcal G$ is the only separable space of almost universal disposition while the Kadec space $\mathcal K$ is the only separable space of almost universal complemented disposition with $1$-FDD, these two spaces represent, in a sense, the same object in different categories. Indeed, if one moves from the category of Banach spaces and isometric embeddings to its ``complemented" analogue, i.e., the category Banach spaces and isometric embeddings admitting a norm one projection, then the separable spaces become the separable spaces with $1$-FDD. It is clear than the ``Gurariy objects" (i.e., the spaces of (almost) universal disposition) become the ``Kadec objects" (the spaces of (almost) universal complemented disposition). Here it is the list of analogies:

\begin{itemize} \item The Gurariy space:
\begin{enumerate}
\item Is a space of almost universal disposition in the category of separable Banach spaces and single arrows (into isometries).
\item It can be obtained as the Fra\"iss\'e limit of separable rational Banach spaces and single arrows.
\item It can be constructed via an $\omega$-times iterated push-out out from a countable dense set of single arrows
    between finite-dimensional Banach spaces.
\item In the category, it is unique, up to isometries.
\item It contains isometric copies of all separable Banach spaces.
\item It is an $\mathcal L_\infty$-space.\medskip
\end{enumerate}

\item The Kadec space:
\begin{enumerate}
\item Is a space of almost universal disposition in the category of separable Banach spaces and double arrows (into isometries
    admitting norm one projections).
\item It can be obtained as the Fra\"iss\'e limit of separable rational Banach spaces with $1$-FDD and double
    arrows.
\item It can be constructed via an $\omega$-times iterated push-out out from a countable dense set of double arrows
    between finite-dimensional Banach spaces.
\item In the category, it is unique, up to isometries.
\item It contains isometric complemented copies of all separable Banach spaces with $1$-FDD. As a by-product,
    it contains isometric copies of all separable Banach spaces.
    \item It is not an $\mathcal L_\infty$-space.
\end{enumerate}
\end{itemize}

In the second part of the paper we introduce and study the notions of space of universal complemented disposition (u.c.d.) and
of space of universal complemented disposition with respect to separable spaces ($\omega$-u.c.d.), their existence
(Every Banach space can be isometrically embedded as a $1$-complemented subspace of a space of $(\omega$-)universal complemented disposition -- cf. Propositions \ref{univfin} and \ref{univomega}), universality and uniqueness properties.\smallskip

The case of $p$-Banach spaces, $0<p<1$ has been treated in a separate paper \cite{ccm} with entirely different techniques.

\section{Almost universal complemented disposition}

All required technical results, definitions and constructions have been gathered in the Appendix section \ref{basic} at the end of the paper. The key notions for this paper are those of double-arrow and almost double arrow, that we state now. A
$(1 + \varepsilon)$-isometry is a linear continuous operator $f : A \to B$ such that for every $x\in A$ verifies $$(1 + \varepsilon)^{-1}\|x\|\leq \|f(x)\|\leq (1 + \varepsilon) \|x\|.$$
We will say that $f$ is contractive if it verifies $(1 + \varepsilon)^{-1}\|x\|\leq \|f(x)\|\leq \|x\|$.\medskip

Given $\alpha, \gamma>1$ and $\beta\geq 0$ a (contractive) $(\alpha,\beta,\gamma)$-arrow is a pair $(f, \overline f)$ of linear continuous operators, $f:A\to B$ and $\overline f:B\to A$ in which $f$ is a (contractive) $\alpha$-isometry, $\|\overline f\| \leq \gamma$ and $\|\overline f f - 1_A\| < \beta$. To simplify some notation, $(1,0,1)$-arrows will be called double arrows, and pairs $(f, \overline f)$ which are $(\alpha, \beta, \gamma)$-arrows for suitable $\alpha, \beta, \gamma $ will be called almost double arrows and depicted as $(f, \overline f): A \leftrightarrows B$.\medskip

Given three almost double arrows $(i_1, \overline i_1) : A \leftrightarrows C$, $(i_2, \overline i_2): A \leftrightarrows B$ and $(i_3, \overline i_3): B \leftrightarrows C$ we will say that the diagram they form

$$\xymatrix{
 A \ar[rr]_{i_1} \ar@<-1ex>[rd]_{i_2} && C \ar@<-1ex>[ll]_{\overline {i_1}} \ar[ld]_{\overline {i_3}}\\
&B \ar@<-0.5ex>[lu]_{\overline {i_2}}  \ar@<-1ex>[ru]_{i_3}}
$$

$\varepsilon$-commutes if $\|i_3i_2 - i_1\| \leq \varepsilon$ and $\|\overline i_2\overline i_3 - \overline i_1\| \leq
\varepsilon$. We will say it almost commutes if there exists $\varepsilon >0$ such that the diagram
$\varepsilon$-commutes. And we will say that it commutes if $i_3i_2 = i_1$ and $\overline i_2\overline i_3 = \overline i_1$.\medskip


\adef A Banach space $E$ will be called of \emph{almost universal complemented
disposition} (a.u.c.d., in short) if for every double arrow $(i, \overline i) : F \leftrightarrows G$ between finite dimensional spaces, every double arrow $(j, \overline j): F \leftrightarrows E$ and every $\varepsilon >0$ there exists a $(1+\varepsilon, \varepsilon, 1)$-double arrow $(J, \overline J): G \leftrightarrows E$ making a commutative diagram

$$\xymatrix{
 F \ar[rr]_{i} \ar@<-1ex>[rd]_{j} && G \ar@<-1ex>[ll]_{\overline i} \ar[ld]_{ J}\\
&E \ar@<-0.5ex>[lu]_{\overline j}  \ar@<-1ex>[ru]_{\overline J}}
$$ \zdef

By Lemma \ref{equiv}, the condition is equivalent to the existence of a $(1+\varepsilon, \varepsilon, 1+\varepsilon)$-double arrow $(J, \overline J): G \leftrightarrows E$ making the diagram $\varepsilon$-commutative. This property essentially corresponds to property [E] of Garbulinska \cite{garbula}, although in that paper only the
almost commutativity of injections is mentioned; the almost commutativity of projections is however used.\medskip

Our immediate purpose is to establish a key Approximation Lemma that will explain the structure of spaces of almost universal complemented disposition. To perturbate projections we will use Lemma \ref{close}, which is modeled upon \cite[Thm. 1.a.9]{lindtzaf}. In order to give an estimate for the distance between projections, a proof is included there. Let us recall that we call a skeleton for $E$ to a sequence $(E_n)$ of  finite-dimensional subspaces so that each $E_n$ is $1$-complemented in $E_{n+1}$ and $E=\overline{\cup_n E_n}$. See section \ref{skeleton} for details.

\begin{lema}[Approximation Lemma]\label{compfelix} Let $\varepsilon <1/3$. If $E$ is Banach space of almost universal complemented
disposition admitting a skeleton then every $(1+\varepsilon, \varepsilon, 1+\varepsilon)$-arrow $(f, \overline f): F\leftrightarrows E$ with $F$ finite-dimensional admits a $(1,0,1)$-arrow $(\phi, \overline \phi): F\leftrightarrows E$ at distance at most $72\varepsilon$.
\end{lema}

\begin{proof} If $(f, \overline {f})$ is a $(1 + \nu, \nu, 1+ \nu)$-arrow then, according to the estimate (3) in Lemma \ref{casiequiv},
$(f/(1 + \nu), \overline {f}/(1 + \nu))$ would be a contractive $(1 +\varepsilon, \varepsilon, 1)$-arrow with $\varepsilon = 3\nu$. Thus, there is no loss of generality assuming that $(f,\overline f)$ is a contractive $(1+\varepsilon, \varepsilon, 1)$-arrow.\medskip

$i)$ \emph{Perturbation step}. Since $E$ has skeleton, it also has a sequence $(E_n)$ of finite dimensional $1$-complemented subspacese so that $E=\overline{\bigcup E_n}$. Let $\imath_n: E_n\to E$ be the isometric embedding with $1$-projection $\overline {\imath_n}: E\to E_n$. The perturbation arguments in Lemma \ref{close} show that if one sets $\e'$ so that $\e'\leq \e\frac{(1+\e)^2}{1-\e}$ and $1+\e'\leq \frac{1-\e^2}{1-3\e}$ then it is possible to find for $\varepsilon<1/3$ and $n$ large enough a $(1+\varepsilon')$-isometry $\tau_{\varepsilon'} : f(F) \to F' \subset E_n$ so that $$\|f - \tau_{\varepsilon'}f\|\leq \varepsilon'$$
 and a projection $p_{\varepsilon'} : E \to F'$ having norm at most $\|p_{\varepsilon'}\|\leq 1 +\varepsilon'$ such that
 \begin{equation}\label{pro}
 \|p_{\varepsilon'} - \tau_{\varepsilon'}f\;\overline f\|\leq \varepsilon'.\end{equation}  A diagram will help to understand the situation
$$
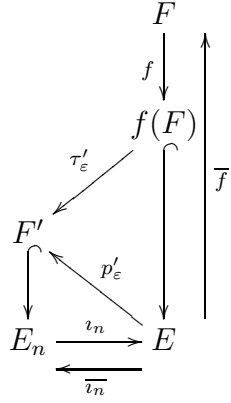
\xymatrix{
&F \ar[d]_{f}\\
&f(F)\ar[ld]_{\tau_\varepsilon'}\ar@{^{(}->}[dd] \\
F' \ar@{^{(}->}[d]\\
{E_n} \ar[r]^{\imath_n} &E \ar[lu]_{p_\varepsilon'} \ar@<-3ex>[uuu]_{\overline f} \ar@<2ex>[l]^{\overline {\imath_n}}}$$

 It is clear that $\|\overline f\; \tau_{\varepsilon'}^{-1}p_{\varepsilon'}\|\leq \frac{1+\varepsilon'}{1-\varepsilon'} \leq 1 + 3\varepsilon' <1+\varepsilon$ and it follows from the estimate (1) in Lemma \ref{casiequiv} that $\tau_{\varepsilon'}f$ is a $(1 + 3\varepsilon)$-isometry. Thus, $(\tau_{\varepsilon'}f, \overline f\; \tau_{\varepsilon'}^{-1}p_{\varepsilon'}i_n): F\leftrightarrows E_n$ is a $( 1 + 3\varepsilon, \varepsilon, 1 + \varepsilon)$-arrow and it follows from the estimate (3) in Lemma \ref{casiequiv}
 that $$(f_1, \overline {f_1}) = \left(\frac{\tau_{\varepsilon'}f}{1 + 3\varepsilon}, \frac{\overline f\; \tau_{\varepsilon'}^{-1}p_{\varepsilon'}i_n}{1 + 3\varepsilon'} \right)$$
is a 
contractive
 $(1 + 6\varepsilon, 6\varepsilon, 1)$-arrow. Moreover, $\|f -  f_1\|\leq 1- \frac{1+\varepsilon'}{1 + 3\varepsilon }\leq \frac{3\varepsilon +\varepsilon'}{1 + 3\varepsilon }\leq 4 \varepsilon$ and, taking into account the estimate (\ref{pro}) above, one gets
$$\|\overline f -  \overline f\; \tau_{\varepsilon'}^{-1}p_{\varepsilon'}\| \leq \|\overline f -  \overline f\; \tau_{\varepsilon'}^{-1}( \tau_{\varepsilon'}f \;\overline f + p_{\varepsilon'} - \tau_{\varepsilon'}f\;\overline f) \| = \|\overline f\; \tau_{\varepsilon'}^{-1}(p_{\varepsilon'} - \tau_{\varepsilon'}f\;\overline f) \| +\varepsilon \leq \frac{\varepsilon}{1-\varepsilon'}$$
which gives
\begin{eqnarray*}\|\overline f - \overline {f_1}\| &=&  \left(1- \frac{1}{1+3\varepsilon'}\right) + \frac{1}{1+3\varepsilon'}\|\overline f -  \overline f\; \tau_{\varepsilon'}^{-1}p_{\varepsilon'}\|\\&\leq& 3\varepsilon' + (1+\varepsilon') \| \overline f -  \overline f\; \tau_{\varepsilon'}^{-1}p_{\varepsilon'}\|\\& \leq& 3\varepsilon' + \frac{(1+\varepsilon')\varepsilon}{1-\varepsilon'}\\&\leq& 4\varepsilon .\end{eqnarray*}

$ii)$ \emph{Correction step.} Apply the correction Lemma \ref{compgarbu} to $(f_1, \overline {f_1}): F\leftrightarrows E_n$ to get a $6\varepsilon$-commutative diagram

$$
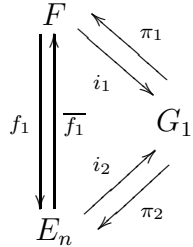
\xymatrix{
 F \ar@<-1ex>[dd]_{f_1} \ar@<0.5ex>[rd]_{i_1} \\
   & G_1\ar@<-2ex>[lu]_{\pi_1} \ar@<2ex>[ld]^{\pi_2}\\
 E_n\ar[uu]_{\overline {f_1}} \ar@<-0.5ex>[ru]^{i_2}}$$
in which $(i_1, \pi_1)$ and $(i_2, \pi_2)$ are $(1,0,1)$-arrows and moreover $\pi_1i_2 = \overline {f_1}$ and $\pi_2i_1 = f_1$.\medskip

$iii)$ \emph{Almost universal complemented disposition step}. Use now that $E$ is of a.u.c.d. to get an $\varepsilon'$-commutative diagram
$$\xymatrix{
 E_n \ar[r]_{i_2} \ar@<-2ex>[d]_{\imath_n} & G_1 \ar@<-1ex>[l]_{\pi_2} \ar@<1.5ex>[ld]_{g_1}\\
E \ar@<0.5ex>[u]_{\overline {\imath_n}}  \ar@<-2.5ex>[ru]_{\overline {g_1}}}
$$in which  $(g_1, \overline {g_1}): G_1\leftrightarrows E$  is a  $(1 + \varepsilon', \varepsilon', 1)$-arrow that extends $(\imath_n , \overline {\imath_n})$. Thus, $(g_1i_1, \pi_1\;\overline {g_1}) : F \leftrightarrows E$ is a $(1 + \varepsilon', \varepsilon',1 )$-arrow. Moreover,
$$ \|g_1 i_1 - f\| \leq \|g_1 i_1 - g_1i_2f_1\| + \|g_1i_2f_1 - f\| \leq  (1 + \varepsilon') 6\varepsilon  + \varepsilon' + \|\imath_n f_1 - f\|
\leq  8\varepsilon + 4\varepsilon = 12\varepsilon$$
And
$$
\|\pi_1 \;\overline {g_1} - \overline f\| \leq \|\pi_1\; \overline {g_1}- \overline {f_1}\;\pi_2\; \overline {g_1} \| + \|\overline {f_1}\;\pi_2 \;\overline {g_1} - \overline {f}\| \leq 6\varepsilon + \|\overline {f_1}\;\pi_2\; \overline {g_1} - \overline {f}\| \leq 6\varepsilon + \varepsilon' + \|\overline {f_1}\; \overline {\imath_n} - \overline {f}\|.$$
From where we get $\|(\pi_1 \overline {g_1} - \overline {f})_{|_{E_n}}\|\leq 12\varepsilon$.\medskip

We have thus obtained that each $(1+\varepsilon, \varepsilon, 1+\varepsilon)$-arrow $F \leftrightarrows E$ can be $36\varepsilon$-approximated by a $(1+\varepsilon', \varepsilon', 1+\varepsilon')$-arrow for any $\varepsilon'\in (0, \varepsilon/3)$ on $E_n$ for $n$ large enough.\medskip

$iv)$ \emph{Ultraperturbation and iteration}. Assume without loss of generality that $E_n = E_1$ in the first step, $E_n=E_2$ in the second step and so on.
We have thus obtained a sequence $(f_n, \overline{f_n})$ of contractive $(1 +\varepsilon_n, \varepsilon_n, 1+\varepsilon_n)$-arrows such that $\|(f_n, \overline{f_n}) - (f_{n+1}, \overline{f_{n+1}})\|\leq 36\varepsilon_n$ on $E_n$. Pick the sequence of $\varepsilon_n$ monotone decreasing so that $\sum \varepsilon_n =\varepsilon$.

We use now a the ultraperturbation argument explained in Lemma \ref{equiv}, with a slight variation since this particular case is simpler. Pick a countably incomplete ultrafilter $\mathscr U$ on $\N$. It is clear that $F_{\mathscr U}\ \leftrightarrows {(E_n)}_{\mathscr U}$ is a $(1, 0, 1)$-arrow at distance $36\varepsilon$ of $(f, \overline f)$ on the whole canonical copy of $E$ inside $E_{\mathscr U}$. The point is that
its image likely does not lie in $E$. We can use then principle of local reflexivity we can push-down this arrow back to $E$ using the argument in Lemma \label{equiv}: given the finite dimensional subspace $[f_n](F)$ pick an $\varepsilon'$-isometry $T_{\varepsilon'}: [f_n](F)\to E$ which is almost a projection and replace the embedding $[f_n]$ by $T_{\varepsilon'}[f_n]$ (see Lemma \label{equiv} for details). The projection thus remains as it was while the inclusion $[f_n]$ is slightly perturbed with the $\varepsilon'$ one prefers so that it takes values in $E$. Call
 $[f_n]'$ this perturbed inclusion to simplify. The new arrow $(u_n, \overline{u_n}) = ([f_n]', [\overline{f_n}]_{|E}): F \leftrightarrows E$  is a $(1 +\varepsilon',  \varepsilon', 1)$-arrow  at distance at most $36\varepsilon + \varepsilon'$ of the original $(f, \overline{f})$ and $\|(u_n, \overline{u_n}) - (u_{n+1}, \overline{u_{n+1}})\|\leq 36\varepsilon_n$ on $E$. Then, both $(u_n)$ and $(\overline{u_n})$ are Cauchy sequences and thus they converge to a $(1, 0, 1)$ arrow $(\phi, \overline{\phi})$ at distance $72\varepsilon$ from $(f, \overline f)$.\end{proof}

An immediate corollary from the Approximation Lemma is:

\begin{proposition}\label{finite} A Banach space of almost universal complemented disposition with skeleton contains isometric $1$-complemented copies of every finite-dimensional Banach space.\end{proposition}

It may seem strange, but we do not know if this result can be obtained without the skeleton assumption. Observe that another reading of Approximation Lemma \ref{compfelix} is that every finite dimensional subspace of a space of almost universal complemented disposition is contained in a finite dimensional $1$-complemented subspace. Thus,
the space has property $\pi_1$.  We obtain now one of the fundamental structural results:

\begin{teor}\label{garbfdd} Every space of almost universal complemented disposition with skeleton contains
isometric $1$-complemented copies of every space with a skeleton.\end{teor}
\begin{proof}  Assume $E$ is a space of almost universal complemented disposition with a $1$-FDD $(E_n)$ having canonical $(1, 0, 1)$-arrows
$(\imath_n, \overline \imath_n):E_n\leftrightarrows E$, and let $Y$ be a space with a skeleton defined by the
sequence of $(1,0,1)$-arrows $(\delta_n, \overline \delta_n): Y_n\leftrightarrows Y_{n+1}$. Assuming without loss of generality that both $Y_0$ and $E_0$ are  of dimension $1$, pick a $(1,0,1)$-arrow $(f_0, \overline f_0): Y_0 \leftrightarrows E_0$. Fix $\varepsilon = \sum \varepsilon_n$ with $(1+\e_{n+1})^2\leq 1+ \e_{n}$.

\begin{itemize}
\item Form first the push-out diagram as in Lemma \ref{amostdpo}:
$$\xymatrix{
Y_0  \ar[r]\ar@<-1ex>[d]_{(f_0,\overline{f_0})}& Y_1\ar@<-1ex>[l]_{(\delta_0,\overline{\delta_0})}  \ar[d]\\
E_0  \ar@<-1ex>[r]_{(\delta'_0,\overline{\delta'_0})} \ar[u]& P_1\ar[l]\ar@<-1ex>[u]_{(f_0',\overline{f_0'})}}$$
which yields $(1,0,1)$-arrows $(\delta_0', \overline{\delta_0'})$ and  $(f_0', \overline{f_0'})$  making the diagram $\varepsilon_1$-commutative (in fact, it is commutative in both directions; i.e., $\delta_0'f_0=f_0'\delta_0$ and $\overline \delta_0 \overline {f'_0} = \overline
f_0 \overline{ \delta'_0}$).\medskip

\item \emph{Inductive step.} Assume that one has obtained an $\varepsilon_n$-commutative diagram

$$\xymatrix{Y_n  \ar[r] \ar@<-1ex>[d]_{(f_n,\overline{f_n})} & Y_{n+1}\ar[d]\ar@<-1ex>[l]_{(\delta_n,\overline{\delta_n})} \\
E_n \ar@<-1ex>[r]_{(\delta'_n,\overline{\delta'_n})} \ar[u]& P_{n+1} \ar[l]\ar@<-1ex>[u]_{(f_n',\overline{f_n'})}}$$ in which $(f_n, \overline {f_n})$ is a contractive $(1+\varepsilon_{n}, 0,  1+\varepsilon_{n})$-arrow, $(\delta_n, \overline{\delta_n})$ and $(\delta'_n,\overline{\delta'_n})$ are $(1,0,1)$-arrows and $(f_n',\overline{f_n'})$ is a contractive $(1+\varepsilon_{n+1}, 0,  1+\varepsilon_{n+1})$-arrow

The a.u.c.d. disposition character of $E$ yields a $(1 + \varepsilon_{n+1}, \varepsilon_{n+1}, 1)$-arrow $(\jmath_n, \overline {\jmath_n}): P_{n+1} \leftrightarrows E$ making a commutative diagram

$$\xymatrix{E_n \ar@<-2ex>[d]_{(\imath_n, \overline{\imath_n})} \ar@<-1ex>[r]_{(\delta'_n,\overline{\delta'_n})} & P_{n+1} \ar@<1.5ex>[ld] \ar[l]\\
E \ar@<0.5ex>[u]  \ar@<-2.5ex>[ru]_{(\jmath_n, \overline {\jmath_n})}}$$

Thus, $(\jmath_n f_n',\; \overline{f_n'}\; \overline{\jmath_n})$ is a $((1 + \varepsilon_{n+1})^2, \varepsilon_{n+1}, 1 + \varepsilon_{n})$-arrow. The approximation lemma \ref{compfelix} yields a $(1,0,1)$- arrow $(g_{n+1}, \overline {g_{n+1}}): Y_{n+1}\leftrightarrows E$ at distance $72\varepsilon_{n}$ of $(\jmath_n f_n',\; \overline{f_n'}\; \overline{\jmath_n})$.

A small perturbation $\tau_{n+1}$, as in Lemma \ref{close}, of $(g_{n+1}, \overline {g_{n+1}})$ yields a $(1+\varepsilon_{n+2}, 0, 1+\varepsilon_{n+2})$ arrow $(f_{n+1}, \overline {f_{n+1}}): Y_{n+1}\leftrightarrows E_{n+1}$ (in which we set $f_{n+1}= \tau_{n+1}g_{n+1}$ and assume that the large $E_\alpha$ is $E_{n+1}$ of course) at a distance $\e_{n+2}$. Form the push-out to get a commutative diagram

\begin{equation}\label{casicasi}\xymatrix{Y_{n+1}  \ar[r] \ar@<-1ex>[d]_{(f_{n+1},\overline{f_{n+1}})} & Y_{n+2}\ar[d]\ar@<-1ex>[l]_{(\delta_{n+1},\overline{\delta_{n+1}})} \\
E_{n+1} \ar@<-1ex>[r]_{(\delta_{n+1}',\overline{\delta_{n+1}'})} \ar[u]& P_{n+2} \ar[l]\ar@<-1ex>[u]_{(f_{n+1}',\overline{f_{n+1}'})}}\end{equation}

According to diagram (\ref{general}) one gets that $(\delta'_{n+1},\overline{\delta'_{n+1}})$ and $(f_{n+1}',\overline{f_n{n+1}'})$ are both contractive $(1+\varepsilon_{n+2}, 0, 1+\varepsilon_{n+2})$ arrows. Induction will be over if the left downward arrow would be contractive. To make it so, we replace $(f_{n+1}, \overline {f_{n+1}})$ by $(\frac{1}{1+\varepsilon_{n+2}}f_{n+1}, (1+\varepsilon_{n+2})\overline{f_{n+1}})$, which is a contractive $((1+\varepsilon_{n+2})^2, 0, (1+\varepsilon_{n+2})^2)$-arrow (in particular a  contractive $(1+\varepsilon_{n+1}, 0, 1+\varepsilon_{n+1})$-arrow) and yields
an $\e_{n+1}$-commutative diagram
\begin{equation}\label{casicasi}\xymatrix{Y_{n+1}  \ar[r] \ar@<-1ex>[d]_{(\frac{1}{1+\varepsilon_{n+2}}f_{n+1}, (1+\varepsilon_{n+2})\overline {f_{n+1}})} & Y_{n+2}\ar[d]\ar@<-1ex>[l]_{(\delta_{n+1},\overline{\delta_{n+1}})} \\
E_{n+1} \ar@<-1ex>[r]_{(\delta_{n+1}',\overline{\delta_{n+1}'})} \ar[u]& P_{n+2} \ar[l]\ar@<-1ex>[u]_{(f_{n+1}',\overline{f_{n+1}'})}}\end{equation}
This concludes the induction. \end{itemize}

Relabel the left downwards arrow as $(f_{n+1}, \overline {f_{n+1}})$ to simplify notation. After this relabeling, we define now the $(1,0,1)$ arrow $(f, \overline f): Y \leftrightarrows E$ we are looking for. Given $y\in Y$ so that $y = \lim_n y_n$ with $y_n\in Y_n$ and $\sum \|y_{n+1} - y_n\|<+\infty$ then we set $$f(y) =
\lim_n f_n(y_n).$$ Since $f_n$ is a $(1+\varepsilon_n)$-isometric embedding, whenever $\lim y_n=0$ then $\lim
f_n(y_n)=0$ and thus $f(y)$ does not depend on the choice of the sequence. To check that $f$ is well defined observe that
\begin{eqnarray*}
\|{f_n}_{|Y_{n-1}} - f_{n-1}\| &=& \|{\tau_n g_n}_{|Y_{n-1}} - f_{n-1}\|\\
&\leq&\|(\tau_n g_n - g_n)_{|Y_{n-1}}\| + \|{g_n}_{|Y_{n-1}} - f_{n-1} \|\\
&\leq&\varepsilon_{n+1} + 72\varepsilon_{n-1} + \|{f_{n-1}'}_{|Y_{n-1}}- f_{n-1}  \|\\
&\leq& \varepsilon_{n+1} + 72\varepsilon_{n-1} + \varepsilon_{n-1}\end{eqnarray*}

and thus, with the proper choice of $(\varepsilon_n)$ the sequence
$(f_n(y_n))$ is Cauchy
\begin{eqnarray*}
\|f_{n+1}y_{n+1} - f_ny_n\|&=&\|f_{n+1}y_{n+1} - f_{n+1}y_n + f_{n+1}y_n - f_ny_n\|\\
&\leq &\|y_{n+1} - y_n\| + (\varepsilon_{n+1} + 73\varepsilon_{n-1})\|y_n\|.
\end{eqnarray*}

The map $f$ is quite clearly an isometric embedding. We define the projection $\overline f$ as follows $$\overline f(e) = \lim_n f_n \overline f_n
\overline \imath_n (e).$$
The operator $\overline f$ is well defined: if  $e=\lim e_n$ with $e_n\in E_n$ and $\sum \|e_{n+1} - e_n\|<+\infty$ then
$\overline f(e) = \lim_n f_n \overline f_n (e_n).$ Observe that   (with a slight abuse of notation)
\begin{eqnarray*}\label{pcauchy}
\|{\overline f_{n+1}}_{|E_n}  - \overline f_{n}\|  &\leq&
\|{\overline f_{n+1}} - \overline{g_{n +1}}\| + \|\overline{g_{n +1}} - \overline f'_n \| + \|\overline f'_n - \overline f_n\|\\
&\leq& \varepsilon_{n+2} + 72\varepsilon_{n+1} + 2\varepsilon_{n}
\end{eqnarray*}
and thus one gets
\begin{eqnarray*}
\| f_{n+1} \overline f_{n+1}  (e_{n+1})  - f_n \overline f_n( e_n)\| &\leq & \| f_{n+1} \overline f_{n+1}( e_{n+1}) - f_{n+1} \overline f_{n+1} (e_n)\| \\
&\quad& + \| f_{n+1} \overline f_{n+1} e_n - f_n \overline f_n e_n\|\\
 &\leq& (1+\varepsilon_{n+1})^2\| e_{n+1} - e_n\| +\\
 &\quad&  + \|f_{n+1}\| \| \overline f_{n+1}(e_n) -  \overline f_n(e_n)\|\\
 &\quad&  +  \| {f_{n+1}}_{|E_n}-  f_n \| \|\overline f_n\|\|e_n\|\\
 &\leq& (1+\varepsilon_{n+1})^2\| e_{n+1} - e_n\|\\
 &\quad& + (1+\varepsilon_{n+1})\left(\varepsilon_{n+2} + 72\varepsilon_{n+1} + 2\varepsilon_{n}\right)\|e_n\|\\
 &\quad& +  \left( \varepsilon_{n+1} + 72\varepsilon_{n} + 2\varepsilon_{n-1}\right)(1+\varepsilon_{n})\|e_n\|
 \end{eqnarray*}
and thus $(f_n \overline f_n (e_n))$ is a Cauchy sequence.\medskip

It  remains to prove that for $y= \lim y_m \in Y$
$$\overline f (f(y)) =\lim_n f_n \overline f_n \overline \imath_n (\lim_m f_m y_m)) = \lim f_m y_m  $$
as it immediately follows form the estimate:
\begin{eqnarray*}
\|\overline f_n(f_{m}(y_{m})) -  y_{m}\| &\leq& \|\overline f_n(f_{m}(y_{m})) -  \overline f_n(f_{n}(y_{n}))\| + \|y_n -
y_m\|\\
&\leq & (1+ \varepsilon_n)\|f_{m}y_{m} - f_{n}y_{n}\| + \|y_n - y_m\|.
\end{eqnarray*}
It is then clear that $\overline f$ is a norm one projection.\end{proof}

In addition to the statement of Theorem \ref{garbfdd}, since $C[0,1]$ contains isometric copies of every separable Banach space and has skeleton, one gets:

\begin{cor}\label{cuniv} Every space of almost universal complemented disposition with skeleton contains isometric copies of every separable Banach space.
\end{cor}

\section{Digression on Banach spaces of almost universal disposition}

Recall (see e.g., \cite{accgm2,accgmLN}) that a Banach space $E$ is said to be of \emph{almost universal disposition} if for every into isometry $i: F \to G$ between finite dimensional spaces, every into isometry $j: F \to  E$ and every $\varepsilon >0$ there exists a $(1+\varepsilon)$-isometry $J: G \to E$ making a commutative diagram

$$\xymatrix{
 F \ar[rr]^{i} \ar[rd]_j&& G \ar[ld]^J\\
&E }$$
Let us show that the approximation lemma remains true in this context.

\begin{lema}\label{felix} If $E$ is a Banach space of almost universal disposition then every contractive $(1+\varepsilon)$-isometry $f: F\to E$ with $F$ finite dimensional admits an isometry $\phi: F\to E$ at distance $3\varepsilon$.\end{lema}

\begin{proof} Let $f: F\to E$ be contractive a $(1+\varepsilon)$-isometry from a finite-dimensional space $F$ into a space $E$ of almost universal disposition. We apply the Correction Lemma
\ref{garblem} to the couple $f: F\to f(F)$ to find another  space $G = G(f, F, f(F))$ and two isometries $i_F:
F\to G$ and $j_F: f(F)\to G$ such that $\|j_Ff - i_F\|\leq \varepsilon$. If  we call $\delta: f(F)\to E$ the
canonical inclusion, the almost universal disposition property of $E$ provides a $(1+\varepsilon')$-isometry $f_G: G\to E$
such that $f_G j_F = \delta$. Hence $f_G i_F: F\to E$ is a
$(1+\varepsilon')$-isometry such that$$\|f_G i _F - \delta  f\| = \|f_G (i _F - j_Ff + j_Ff ) -
\delta  f \|\leq
\varepsilon\|f_G\| + \|f_G j_Ff - \delta  f \| = \varepsilon(1+\varepsilon').$$ Thus, every $(1+\varepsilon)$-isometry $f$ admits a
$(1+\varepsilon/2)$-isometry $f_1$ $2\varepsilon$-close. So $f_1$ admits a $(1+\varepsilon/4)$-isometry $f_2$ at distance
at most $\varepsilon$ and we thus obtain a sequence $f_n$ of $(1 + \varepsilon/{2^n})$-isometries so that
$\|f_n - f_{n-1}\|\leq 3\varepsilon/{2^{n-1}}$. In particular, $(f_n(x))$ is a Cauchy sequence for every $x\in F$.
The map $\phi : F\to E$ given by
$$\phi(x) =  \lim f_n(x)$$ is an into isometry and $\|f - \phi\|\leq 3\varepsilon$.\end{proof}
F. Cabello  suggested to us that Lemma \ref{felix} could be true. An immediate (and well-known) consequence of Lemma \ref{felix} is:

\begin{cor}\label{thesemiedge} Every space of almost universal disposition contains isometric copies of all separable spaces.
\end{cor}
\begin{proof} Let $X$ be a separable space, which we write as the closure of the union $X=\overline{ \bigcup_{n=0}^\infty X_n}$ of
a sequence of finite-dimensional spaces $X_n$. Assume that $X_0$ is one dimensional. Let $E$ be  a space of almost universal
disposition and thus, fixing $\varepsilon>0$, any isometric embedding $f_0: X_0\to E$ can be extended to a
$(1+\varepsilon)$-isometry $f_0': X_1\to E$ which, by Lemma \ref{felix}, admits an isometric embedding $f_1:X_1\to E$ at distance
$2\varepsilon$. Now it is $f_1$ which admits a $(1+\varepsilon/2)$-isometric extension $f_1':X_2\to E$ which, by the lemma, admits an
isometric embedding $f_2:X_2\to E$ at distance $\varepsilon$. Continue in this way and define $f:\bigcup_n X_n\to E$ as
$f(x)=\lim f_n(x)$. This is an isometric embedding that extends to an isometric embedding $X\to E$ as desired. \end{proof}

The assertion in Corollary \ref{thesemiedge}  was proved by Gurariy \cite{gurari} for his space and by Gevorkyan \cite{gevor} in full generality.

\section{Construction of separable spaces of almost universal complemented disposition}

We show now that the basic construction device as presented in \cite{accgm2,accgmLN,castsua}, and used in \cite{cgk}, that provided a
unified method to construct spaces of (almost) universal disposition, such as the Gurariy, Kubi\'s or the $\mathcal
L_\infty$-envelopes, can be adapted to construct separable spaces of almost universal complemented disposition.

\begin{teor}\label{KX} Every Banach space $X$ with separable dual can be isometrically embedded as a 1-complemented subspace of
a separable space $\mathscr K(X)$ of almost universal complemented disposition.
\end{teor}

\begin{proof} Let $\mathfrak U= \{(u, \overline u): F_u \leftrightarrows  G_u\}$ be a countable set of double arrows between finite
dimensional spaces as in Lemma \ref{denseness}. The space $F_u$ will be called the domain of $u$ and $G_u$ its codomain. We will call ${\mathrm dom} \mathfrak U$ the set of the domains of the elements $u$ so that $(u, \overline u)\in \mathfrak U$. For fixed $F_u\in {\mathrm dom} \mathfrak U $ and $X$ with separable dual, any subspace of $\mathfrak L(F_u, X) \oplus \mathfrak L(X, F_u)$ is separable. Thus, for $m\in \N$, let $L_m(F_u, X)$ be a countable dense subset of the space of all contractive $(1+2^{-m}, 0, 1+2^{-m})$-arrows $F_u \leftrightarrows X$. Form now the countable set

$$D(E) = \;\bigcup_{u\in \mathfrak U}  \;\bigcup_{m\in \N} L_m(F_u, E).$$

We start fixing an enumeration $\{d_{0,j} , j\in \N\}$ of  $D(X)$. To avoid ambiguities, the map $F_{d_{0,1}} \to G_{d_{0,1}}$, which should be
called $d_{0,1}$ by the general convention above, will be called $v_0$. Our first step is then to form the push-out
$$\begin{CD}
F_{d_{0,1}} @>{v_0}>> G_{d_{0,1}} \\
@V{d_{0,1}}VV @VV{{d_{0,1}}'}V\\
X@>>{u_0}> P_1
\end{CD}$$
in which $u_0$ is an isometric embedding by Lemma \ref{isom}. Observe that $P_1/X =  G_{d_{0,1}}/F_{d_{0,1}}$ is finite-dimensional.\medskip

Assume now that spaces $P_1, \dots, P_n$ having separable dual  have already been obtained together with into isometries $u_k: P_{k}\to P_{k+1}$ so that one can assume that $P_k$ is a subspace of $P_{k+1}$
and enumerations  $\{d_{k,j} , j\in \N\}$ of  $D(P_k)$ have also been fixed. Let us call $I_{n+1} = \{d_{i, j}: i+j \leq n+1\}$ and form the push-out

$$\begin{CD}
\ell_1(I_{n+1}, F_u)@>{\oplus u}>> \ell_1(I_{n+1}, G_u)\\
@V{\sum d}VV @VVV\\
P_n@>>{u_n}> P_{n+1}.
\end{CD}$$
in which $\oplus u$ is the natural (into isometry) amalgamation of the maps $F_u\to G_u$ that appear involved in $I_{n+1}$ and $\sum d$ is the (contractive) operator sum of the operators in $I_{n+1}$. This, again by Lemma \ref{isom}, makes $u_{n}$ an into isometry. Since $P_{n+1}/P_n= \ell_1(I_{n+1}, G_u)/\ell_1(I_n, F_u)$ is finite-dimensional, as well as $P_1/X$ and $X$ has separable dual, all $P_n$ have separable dual and the process can be actually performed and the space $\mathscr K(X) = \overline{\cup P_n}$ is separable. Let us show that:

\begin{itemize}
\item $\mathscr K(X)$ contains an isometric $1$-complemented copy of $X$.
\item $\mathscr K(X)$ is a space of almost universal complemented disposition.
\end{itemize}

The first part follows from the ``Moreover" part of lemma \ref{amostdpo}, which says that the $u_{n}:P_{n}\to P_{n+1}$ maps are actually part of certain $(1,0,1)$-arrows $(u_n, U_n)$; which means that each $P_n$ is $1$-complemented in $P_{n+1}$ and therefore $X$ is $1$-complemented in $\mathscr K(X)$. To prove the almost universal complemented disposition of $\mathscr K(X)$, fix $0<\varepsilon<1/3$ and consider a double arrow $(\delta, \overline \delta): F\leftrightarrows G$ between two finite dimensional spaces and a double arrow $(f, \overline f): F\leftrightarrows \mathscr K(X)$. Assume that the subspace $A= f(F)$ of $\mathscr K(X)$ is $n$-dimensional. $A$ is complemented by a projection $f\overline f$ of norm $C=1$. Pick $\{a_1,\dots, a_n\}$ a basis for $A$ so that $ \mathrm{dist} (A, \ell_1^n)^{-1} \sum |\lambda_i|\leq \|\sum \lambda_ia_i\|\leq
\sum|\lambda_i|$. Fix $k$ large enough so that  $2^{-k}< \varepsilon\mathrm{dist} (F, \ell_1^n)^{-1}$, there exist $x_i\in P_{n_k}$, $i\in \{1,\cdots,n\}$, such that $\|x_i-a_i\|\leq 2^{-k}\leq \frac{\varepsilon}{\mathrm{dist} (F, \ell_1^n)}$ and assume without loss of generality $P_k=P_{n_k}$.\medskip

By Lemma \ref{close}(1), the map $\tau a_i=x_i$ is a $(1+\varepsilon)$-isometry $A \to [x_i]$ and thus $\tau f :  F \to P_k$ is a $(1+\varepsilon)$-isometry with range $[x_i]$ and complemented via some projection $p'$ of norm
$\frac{1-\varepsilon^2}{1-3\varepsilon}$ for which $\|p' - \tau f\overline f\|\leq \frac{\varepsilon(1+\varepsilon)^2}{1 -
\varepsilon}$, besides $\|f - \tau f\|\leq
    \varepsilon$. Set $f_\varepsilon = \frac{\tau f}{1 + \varepsilon}$ and $\overline f_\varepsilon = (1 + \varepsilon)\overline{f}p'$ so that, again by Lemma \ref{close}(2),  $(f_\varepsilon, \overline{f_\varepsilon}): F \leftrightarrows P_k$ is a contractive $(1+3\varepsilon, 0, \frac{1+3\varepsilon}{1-3\varepsilon})$-arrow, $\|f - f_\e\|\leq
    2\varepsilon$ and
$$\|\overline{f_\e}-f_\e \overline{f}\|\leq 3\e$$

Using Lemma \ref{denseness} we pick then $(u,\overline u): F_u\leftrightarrows G_u$ in $\mathfrak U$ for which there exist surjective contractive $(1+2^{-k})$-isometries $\alpha, \beta$ so that the square
$$\begin{CD}
F_u@>u>> G_u\\
@V\alpha VV @VV\beta V\\
F@>>{\delta}> G\end{CD}$$
is commutative in both directions, i.e., $\delta\alpha = \beta u$ and $\alpha \overline u  = \overline {\delta} \beta$. Thus, $(f_\varepsilon \alpha, \alpha^{-1} \overline{f_\varepsilon}): F_u\leftrightarrows  P_k$ is a contractive $(1+3\e, 0, \frac{1+3\varepsilon}{1-3\varepsilon})$-arrow. Thus, some contractive $(1+3\varepsilon, 0, \frac{1+3\varepsilon}{1-3\varepsilon})$-arrow $(f', \overline{f'}): F_u\leftrightarrows P_k$ at distance $\varepsilon'$ (to be chosen) must exist in some set $L_m(F_u, P_k)$, which means that $(f', \overline{f'}) = d_{k, s}$ for some $s$. Since $d_{k, s}\in I_{k+s}$ we have that $(f', \overline{f'})$ is one of the elements forming the operator $\sum d$ that appears in the push-out diagram
$$\begin{CD}
\ell_1(I_{k+s}, F_u)@>{\oplus u}>> \ell_1(I_{k+s}, G_u)\\
@V{\sum d}VV @VVd'V\\
P_{k+s}@>>{u_{k+s}}> P_{k+s+1}
\end{CD}$$

According to Lemma \ref{POprojection}, $(f', \overline{f'})$ admits a contractive $\left(1+3\e, 0, \frac{1+3\varepsilon}{1-3\varepsilon}\right)$-arrow extension
$(f'', \overline{f''}): G_u\leftrightarrows P_{k+s+1}$. The composition $(f'' \beta^{-1}, \beta \; \overline{f''} )$ is a, say,  $(1+ 7\varepsilon, 0, \frac{1+7\e}{1- 7\e})$-arrow such that
\begin{eqnarray*}\|{f''\beta^{-1}}\delta - f \| &=& \|f''u \alpha^{-1} - f \|\\ &=& \|f'\alpha^{-1} - f\|\\ & =& \|(f' - f_\varepsilon \alpha+
f_\varepsilon\alpha)\alpha^{-1} - f\|\\
&\leq&\varepsilon' + \|f_\varepsilon - f\|\\
&\leq&\varepsilon' + 2\varepsilon
\end{eqnarray*}
and also
\begin{eqnarray*}\|\left(\overline{\delta}\beta \; \overline{f''}  - \overline{f}\right)_{|P_{s+k}} \| &=& \|\left(\alpha\; \overline u \; \overline{f''} -  \overline{f}\right)_{|P_{s+k}} \|\\
 &=& \|\left( \alpha \; \overline{f'} \overline{ u_{k+s}} - \overline f\right)_{|P_{s+k}} \|\\& =&
\|\left(\alpha \left (\overline{f'} - \alpha^{-1} \overline{f_\varepsilon} + \alpha^{-1} \overline{f_\varepsilon}\right )\overline {u_{k+s}} -  \overline{f}\right)_{|P_{s+k}} \|\\
&\leq& \e' + \|\left(\overline{f_\varepsilon} -  \overline{f}\right)_{|P_{s+k+1}} \|\\
&\leq&\e' + 3\e.
\end{eqnarray*}

We have thus obtained that

$$\forall \e>0 \quad \exists n\in \N \; \; \exists (f_n, \overline f{_n}): G \leftrightarrows  P_n : \|f_n \delta- f\|\leq 4\e;\quad \mathrm{and}\quad \|\overline \delta \overline f{_n} - \overline f_{|P_{n-1}}\|\leq 4\e$$
and $(f_n, \overline f{_n})$ is a $((1+7\varepsilon), 0 , \frac{1+7\e}{1- 7\e})$-arrow. An ultraperturbation argument we sketch now is sufficient to conclude that for each $\e$ there is a $((1+\varepsilon), 0 , (1+ \varepsilon))$-arrow $G\leftrightarrows \mathscr K(X)$ making the diagram $\e$-commutative, which is condition i) in Lemma \ref{equiv}, and therefore $\mathscr K(X)$ is a space of almost universal complemented disposition.\medskip

\emph{Ultraperturbation argument}: Observe that the problem lies in that the projection $\overline{f_m}$ behave well only on $P_m$. Inclusions behave well in the sense that once some $f_n$ has been obtained then one can set $f_m = u_{m-1}\dots u_n f_n$. To get a good projection defined on the whole $\mathscr K(X)$ just define $[\overline f{_n}]:  (\Pi P_n)_\U \leftrightarrows G_\U = G$ and compose with the diagonal canonical embedding $\mathscr K(X)\to \mathscr K(X)_\U$.\end{proof}

\section{Uniqueness}
\adef Given a Banach space $X$ with separable dual we will denote $\mathscr K(X)$ the space constructed in Theorem \ref{KX}.\zdef

We need a simple observation:

\begin{lema}\label{baplemma} If $X$ has separable dual and skeleton then $\mathscr K(X)$ has skeleton.
\end{lema}
\begin{proof} Using the enumeration of Theorem \ref{KX} one gets that $X$ is $1$-complemented and has finite codimension in $P_1$, and then $P_n$ is $1$-complemented and has finite codimension in $P_{n+1}$. Let us write $P_{n+1} = P_n \oplus C_n$ and $P_1= X \oplus F$ with $F$ finite-dimensional. If $(X_n)_n$ is a skeleton of $X$ then
$(X_n \oplus C_n)$ is a skeleton for $\mathscr K(X)$.\end{proof}

Thus, contrarily to what occurs with Gurariy space:

\begin{proposition}\label{noisomo}
There are non-isomorphic separable spaces of almost universal complemented disposition.
\end{proposition}
\begin{proof} When $X$ has not the BAP the space $\mathscr K(X)$ cannot have skeleton (it cannot have BAP) and thus it cannot be isomorphic to any space $\mathscr K(Y)$ constructed over a space $Y$ with skeleton by virtue of the previous lemma\end{proof}

This marks a neat difference with the situation for separable spaces of almost universal
disposition. Still, there is only one space of almost universal complemented disposition with skeleton, up to isomorphism: on one side the class of separable spaces with BAP is closed under $c_0$-sums, which means by Lemma \ref{mix} that there is only one complementably universal member, up to isomorphisms; since all spaces of complementably universal disposition with skeleton are complementably universal for the class of separable spaces with BAP, by Theorem \ref{garbfdd}, the assertion follows. Let us show that the space is unique, up to isometries

ameno
\begin{teor}\label{uniqGarbu} Let  $U, V$ be two spaces of almost universal complemented
disposition having a skeleton. let $\imath: A \to B$ be an isometry between two finite-dimensional $1$-complemented subspaces $A\subset U$ and $B\subset V$. For every $\e >0$ there exists an isometry $\tau: U \to V$ such that $\|\tau_{|A} - \imath\|\leq \e$. In particular, all spaces of almost universal complemented disposition with skeleton are
isometric.\end{teor}
\begin{proof} The proof is a simple back-and-forth argument combining the Approximation Lemma \ref{compfelix} and the Perturbation argument of Lemma \ref{close}: let  $(U_n)$ (resp. $(V_n)$) be a skeleton for $U$ (resp. $V$), so that $(u_n, \overline{u_n}): U_n\leftrightarrows U$ and $(v_n, \overline{v_n}): V_n\leftrightarrows V$ are double arrows.

Let  $A\subset U$ be a $1$-complemented subspace with embedding and projection $(\imath_A, \overline {\imath_A})$; and, analogously, $B\subset U$ be a $1$-complemented subspace with embedding and projection $(\imath_B, \overline{ \imath_B})$. Let $\imath: A \to B$ be an isometry between them. Set $\varepsilon=\sum \e_n$. After some $\e'$-perturbation of $(\imath_A, \overline {\imath_A})$ using Lemma \ref{close} ---which we do not relabel--- we can assume that $[\imath_A (A) +U_1] \subset U_{2}$ (actually some $n_2$, but again we do not relabel). This small perturbation we ignore by using the Approximation Lemma \ref{compfelix}, so we still assume that that $(\imath, \imath^{-1}\overline{\imath_B})$ is a double arrow $A \leftrightarrows V$. By the a.u.c.d character of $V$ this double arrow extends to some $\e_1$-arrow $U_2 \leftrightarrows V$ that can therefore be $\e_1$-approximated by a double arrow: $(\imath_1, \overline{\imath_1}): U_2 \leftrightarrows V$. Now repeat the argument back: we work with $(\jmath_1, \overline{\jmath_1})= (\imath_1^{-1}, \imath_1\overline{u_2}):  \imath_1(U_2) \leftrightarrows U$. After some $\e_2$-perturbation we assume that  $\imath_1(U_2) \subset V_2$, use the Approximation Lemma \ref{compfelix} to not relabel, so that
 $(\jmath_1, \overline{\jmath_1})$ is a double arrow that the a.u.c.d. character of $U$ allows one to extend to an $\e_2$-arrow $V_2 \leftrightarrows U$ that can therefore be $\e_2$-approximated by a double arrow: $(\jmath_2, \overline{ \jmath_2}): V_2 \leftrightarrows U$. And forth again. Iterate the argument. \end{proof}

This result should be compared to \cite[Thm. 7.3]{garbula}. We will (improperly) call Kadec space to $\mathscr K(\R)$, the only (up to isometries) separable space of almost universal complemented disposition having skeleton. Which is of course complementably universal for all separable spaces with BAP. We say ``improperly" because we cannot prove that the Kadec space $\mathcal K$ constructed in \cite{kade} is of almost universal complemented disposition, although we know that it is isomorphic to $\mathscr K(\R)$.

\section{Spaces of universal complemented disposition}

Spaces of universal disposition (i.e., the case $\varepsilon=0$) were studied in \cite{gurari,accgm2,accgmLN}. In the same spirit, we have:

\adef A Banach space $E$ will be called of universal complemented disposition if given a double
arrow $(i, \overline i) : F \leftrightarrows G$ between finite dimensional spaces and a double arrow $(j, \overline j ): F
\leftrightarrows E$ there exists a double arrow $(J, \overline J): G \leftrightarrows E$ making  a commutative diagram.
$$\xymatrix{
 F \ar[rr]_{i} \ar@<-1ex>[rd]_{j} && G \ar@<-1ex>[ll]_{\overline i} \ar[ld]_{J}\\
&E \ar@<-0.5ex>[lu]_{\overline j}  \ar@<-1ex>[ru]_{\overline J}}
$$\zdef

The additional hypothesis of having separable dual is no longer required, and one gets:

\begin{proposition}\label{univfin} Every Banach space can be isometrically embedded as a $1$-complemented subspace of a space of universal complemented disposition.
\end{proposition}
\begin{proof} We will use the same device as for the construction of $\mathscr  K(X)$, although everything is much simpler now
since no correction lemmata or countable dense sets are required. The construction has now $\omega_1$ steps.  At step $\alpha$, assuming $P_\alpha$ has been obtained, we get $P_{\alpha+1}$ as the push out in the diagram
\begin{equation}\label{palfa}\begin{CD}
\ell_1(I_\alpha, F_u)@>{\oplus u}>> \ell_1(I_\alpha, G_u)\\
@V{\sum d}VV @VVV\\
P_\alpha@>>{u_\alpha}> P_{\alpha+1}.
\end{CD}\end{equation}
Here $I_\alpha$ represents an index set containing: all $(1,0,1)$-arrows $d: F_u \leftrightarrows P_\alpha$ from a finite dimensional space $F_u$ into $P_\alpha$ each of them repeated as many times as $(1,0,1)$-arrows $u: F_u\leftrightarrows G_u$ between finite dimensional spaces. The operator $\oplus u$ is the vector sum of all operators $u$ and $\sum d$ the sum of all operators $d$. For $\alpha=0$ set $P_\alpha=X$. If $\alpha$ is a limit ordinal then $P_\alpha = \overline{\cup_{\beta<\alpha} P_{\beta }}$.\medskip

The resulting space  $\mathscr  K_{\omega_1}(X)$ is of universal complemented disposition. Indeed, consider a double arrow $(\delta, \overline \delta): F\leftrightarrows G$ between two finite dimensional spaces and a double arrow $(f, \overline f): F\leftrightarrows \mathscr K_{\omega_1}(X)$. We choose $\alpha<\omega_1$ in such a way that
$f(F)$ is actually contained in $P_\alpha$. So, $(f, {\overline f}_{|P_\alpha}): F\leftrightarrows P_\alpha$ is one of the arrows $d$ appearing in diagram (\ref{palfa}) and can therefore be extended through any double arrow $F\leftrightarrows G$, in particular $(\delta, \overline \delta)$ to a double arrow $G\leftrightarrows P_{\alpha+1}$. We have obtained now that
$$\exists \beta \;\;\forall\; \alpha\geq \beta \quad \exists (f_\alpha, \overline{f_\alpha})\ : G \leftrightarrows P_{\alpha+1}: f_\alpha\delta =f \quad\mathrm{and}\quad
\overline{\delta}\;\overline{f_\alpha} = \overline{f_{P_\alpha}}$$
which a simple ultraperturbation argument transforms into a double arrow $G\leftrightarrows \mathscr K_{\omega_1}(X)$ extending $(f, \overline f)$.\end{proof}

\adef Given a Banach space $X$ we will call $\mathscr K_{\omega_1}(X)$ the Banach space constructed in Theorem \ref{univfin}.\zdef

\begin{proposition}\label{underCHfin}$\;$\begin{enumerate}
\item Under {\sf CH}, if $X$ has an $\omega$-skeleton then the space $\mathscr K_{\omega_1}(X)$ has an $\omega$-skeleton.
\item Under {\sf CH}, if $X$ has skeleton then $\mathscr K_{\omega_1}(X)$ has the BAP.
\end{enumerate}
\end{proposition}
\begin{proof} Under {\sf CH} ($\mathfrak c = \aleph_1$) an thus a set of size $\mathfrak c$ can be written as an increasing union of $\omega_1$ countable sets. Now, the set we are considering is that of double arrows between two separable spaces $A, B$, which has the size of  $\mathfrak L(A,B) \oplus \mathfrak L(B,A)$, namely $\mathfrak c^{\aleph_0} = \mathfrak c$. Let $(X_\alpha)_{\alpha<\omega_1}$ be the $\omega$-skeleton of $X$. Proceed as in the proof of Proposition \ref{univfin} except that now we can arrange things so that all $P_\alpha$ are separable. This is
done by representing each size $\mathfrak c$ set $I_\alpha$ (since $X_\alpha$ is countable) as an increasing union $I_\alpha=\cup_{\mu<\omega_1}\Gamma_{\alpha,\mu}$ of countable sets and performing successive ``diagonalizations" of those sets so that each push-out is done using only with a countable number of operators, so that the resulting push-out space is separable.
Indeed, start with $P_0=X_0$ and decompose the first set $I_0=\cup_{\mu<\omega_1}\Gamma_{0,\mu}$ as an increasing union of countable sets. Make the first push-out only with the elements of $\Gamma_{0,1}$. The space $P_1'$ is thus a separable superspace of $X_0$. Make a new push-out
$$\begin{CD}
X_0@>>> P_1'\\
@VVV @VVV\\
X_1@>>> P_1\end{CD}$$
to obtain a new separable enlargement of $X_1$. Decompose now the set $I_1=\cup_{\mu<\omega_1}\Gamma_{1,\mu}$ as an increasing union of countable sets $\Gamma_{1,\mu}$ and make now push-out only with the elements of $\Gamma_{0,2}\cup \Gamma_{1,1}$. The new space thus obtained $P_2'$ is a separable enlargement of $P_1$. Assume now that a separable $P_\alpha'$ has already been obtained, make the new push-out
$$\begin{CD} X_\alpha@>>> P_\alpha'\\
@VVV @VVV\\
X_{\alpha+1}@>>> P_\alpha\end{CD}$$
and write now $I_\alpha=\cup_{\mu<\omega_1}\Gamma_{\alpha, \mu}$ as the increasing union of countable sets $\Gamma_{\alpha, \mu}$ and make push-out only with the elements of $\cup_{i+j\leq \alpha+1}  \Gamma_{i, j}$. This yields a separable  $P_{\alpha +1}'$. The skeleton of $\mathscr K_{\omega_1}(X)$ are the spaces $(P_\alpha)_{\alpha<\omega_1}$.\medskip

To prove (2) we will actually show that the construction can be modified so that for each scountable $\alpha$ the push-out space $P_\alpha$ has a skeleton. Let us simplify the notation assuming that the space $P_\alpha$ has been obtained making push-out with the countable set $I_\alpha$. Decompose $I_\alpha$ into an increasing sequence of finite sets $I_\alpha = \cup_m F_m$ and observe that $P_\alpha$ could have been obtained making just a sequence of iterated push outs starting with $X$: at step $m$ make push-out with only the elements of $F_m$. Next, observe that the real content of Lemma \ref{baplemma} is that when $X$ has skeleton then so does the space $P_\omega$. Thus, $P_\alpha$ has skeleton. This immediately implies that $\mathscr K_{\omega_1}(X)$ has the BAP since any of its finite dimensional subspaces is contained into some $P_\alpha$, and a Banach space such that any finite dimensional subspace is contained into a $\lambda$-complemented subspace with the $\lambda$-BAP must have the $\lambda$-BAP.\end{proof}

The general version of Theorem \ref{garbfdd} becomes:

\begin{proposition}\label{garbskel}
A Banach space of universal complemented disposition that has a $\omega$-skeleton formed by spaces with skeleton contains isometric $1$-complemented copies of every Banach space with a skeleton.\end{proposition}

\noindent \emph{Sketch of proof}. Let $Y$ be a space with skeleton and let $E$ be a space of universal complemented disposition having a $\omega$-skeleton of spaces $E_\alpha$ so that each $E_\alpha$ admits a skeleton $E_{\alpha,n}$. The uncountable cofinality of $\omega_1$  will make he image of $Y$ obtained be lying in some of the separable spaces $E_\alpha$; and since these are $1$-complemented in $E$ the point is to obtain the copy of $Y$ complemented in some $E_\alpha$.

Let us proceed as in the proof of Theorem \ref{garbfdd}. A close examination of that proof reveals that the major part of the difficulties and the hard work in that proof was to get the image of $Y$ complemented, something that could be done because $E$ had skeleton: so one just needed to inductively add one by one the pieces of the skeleton to finally get a projection defined on all of $E$.

What has to be done now is to inductively add, one step each, finite dimensional pieces $E_{\alpha_n, m_n}$ with $E_{\alpha_i, m_i}\subset E_{\alpha_j, m_j}$ for $i<j$ in such a way that $E_{\sup \alpha_n} = \overline {\cup_n E_{\alpha_n, m_n}}$. This would provide $Y$ complemented
in $E_{\sup \alpha_n}$ and the argument is complete.\hfill $\square$\medskip

The assumption ``having a $\omega$-skeleton" is necessary. To show this, let us consider a different way to obtain spaces of universal complemented disposition: Let $Z$ be a space of almost universal complemented disposition and let $\U$ be a countably incomplete ultrafilter on $\N$. The ultrapower $Z_\U$ is quite obviously a space of universal complemented disposition. In particular, one thus has:

\begin{proposition}\label{example} Let $X$ be a dual separable Banach space. The space $\mathscr K(X)_\U$ is a space of universal complemented disposition.
\end{proposition}

Recall that a Banach space $X$ is said to have the Uniform Approximation Property (UAP) \index{Uniform Approximation Property, UAP} when every ultrapower of $X$ has the BAP. It is clear that the UAP exactly means that $X$ has the $\lambda$-AP and there exists a ``control function'' $f:\N\to \N$ so that, given $F$ and $\lambda'>\lambda$, one can choose $T$ such that  $\text{rank}(T) \leq  f(\dim F)$ and $Tf=f$ for all $f\in F$, with $\|T\|\leq\lambda'$. Since $X^{**}$ is complemented in some ultrapower of $X$, when $X$ has the UAP then all even duals have the UAP. And since approximation properties pass from the dual to the space, when $X$ has the UAP all its duals have the UAP. See \cite[Section 7]{casazza} for details. Therefore, Banach spaces with the BAP but whose duals do not have the BAP
(see \cite[Section 7]{casazza}) admit ultrapowers without the BAP. \medskip

Regarding uniqueness, there are at least two (three under {\sf CH}) non isomorphic spaces of universal complemented disposition:

\begin{proposition}\label{example2}\label{underCHfin}$\;$\begin{enumerate}
\item  The spaces $\mathscr K(c_0)_\U$ and $\mathscr K_{\omega_1}(c_0)$ are not isomorphic.
\item  Under {\sf CH}, the spaces $\mathscr K(\R)_\U$, $\mathscr K_{\omega_1}( \R)$ and $\mathscr K_{\omega_1}(X)$ for $X$ a separable Banach space without BAP are not isomorphic
\end{enumerate}
\end{proposition}
\begin{proof}  In \cite{accgm3} it was proved that infinite dimensional ultrapowers never contain complemented copies of $c_0$, and thus $\mathscr K(X)_\U$ cannot contain $c_0$ complemented. Since any copy of $c_0$ must be complemented in any space with $\omega$-skeleton, thanks to Sobczyk's theorem, assertion (1) is clear. The space $\mathscr K(\R)_\U$ cannot have $\omega$-skeleton nor the BAP: otherwise, every ultrapower $X_\U$ of a separable space with BAP should have the BAP, which is false. The space $\mathscr K_{\omega_1}( \R)$ has the BAP and $\omega$-skeleton; and the space $\mathscr K_{\omega_1}(X)$ has $\omega$-skeleton but not BAP. All this proves (2).\end{proof}

Regarding universality results, observe that

 \begin{cor} No Banach space with $\omega$-skeleton can be universal for the class of spaces with density character $\aleph_1$.
 \end{cor}\begin{proof} As it has been said, every copy of $c_0$ must be complemented in a space with $\omega$-skeleton; and thus, spaces with density character
$\aleph_1$ but containing uncomplemented copies of $c_0$ cannot embed in a space with $\omega$-skeleton.\end{proof}

Thus, there is no point in asking if a space of universal complemented disposition contains isometric copies of all spaces with density character
at most $\aleph_1$ (since one must exclude those with $\omega$-skeleton). It is quite curious that the spaces $K_\U$ with $K$ of almost universal complemented disposition contain, at least under {\sf CH}, isometric copies of all spaces with density character $\aleph_1$: indeed, $K$ contains $C[0,1]$, hence $K_\U$ contains  $C[0,1]_\U$ which is, under {\sf CH}, isomorphic to $\ell_\infty/c_0$ by \cite[Proposition
    2.4.1]{bank}; and this last space is universal for all spaces with density character $\aleph_1$ by Parovi\v{c}enko's theorem
    (\cite{paro}, \cite[p.81]{walker}). A different thing is to ask if space of universal complemented disposition must contain isometric copies of all spaces with $\omega$-skeleton. See Proposition \ref{partanswer} and Problem (6).

\section{Spaces of universal complemented disposition for separable spaces}

In the same way that the notion of space of universal disposition can be extended to ``space of universal disposition with respect to the class of separable spaces", we can define:

\adef A Banach space $E$ will be called of $\omega$-universal complemented disposition if given a double
arrow $(i, \overline i) : S_1 \leftrightarrows S_2$ between separable spaces and a double arrow $(j, \overline j): S_1
\leftrightarrows E$ there exists a double arrow $(J, \overline J): S_2 \leftrightarrows E$ making a commutative diagram
$$
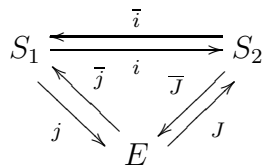
\xymatrix{
 S_1 \ar[rr]_{i} \ar@<-1ex>[rd]_{j} && S_2 \ar@<-1ex>[ll]_{\overline i} \ar[ld]_{\overline J}\\
&E \ar@<-0.5ex>[lu]_{\overline j}  \ar@<-1ex>[ru]_{J}}
$$
\zdef

One has:

\begin{proposition}\label{univomega} Every Banach space can be isometrically embedded as a $1$-complemented subspace of a space of $\omega$-universal
    complemented disposition.
\end{proposition}

The construction is immediate after that in Proposition \ref{univfin} just replacing ``finite dimensional" by ``separable". Let us call
$\mathscr K_{\omega_1}^{\mathfrak S}(X)$ the resulting space. It is of $\omega$-universal complemented disposition exactly as in the proof of Proposition \ref{univfin}, which remains valid since
no countable set is cofinal in $\omega_1$, and thus any operator from a separable space into $\mathscr K_{\omega_1}^{\mathfrak S}(X)$ actually has
its image contained in some space $P_\alpha$ for some $\alpha<\omega_1$. The $\omega$-version of Theorem \ref{uniqGarbu} is:

\begin{teor}\label{uniqueomega}\label{uniqsep} Let  $U, V$ be two spaces of $\omega$-universal complemented
disposition having a $\omega$-skeleton. Let  $\imath: A \to B$ be an isometry between two separable $1$-complemented subspaces $A\subset U$ and $B\subset V$. There exists an isometry $\tau: U \to V$ such that $\tau_{|A} = \imath$. In particular, all spaces of $\omega$-universal complemented disposition with $\omega$-skeleton are isometric.\end{teor}
\begin{proof} The proof is much simpler than that of Theorem \ref{uniqGarbu} since no approximation of perturbation is required: just a straightforward back-and-forth argument. \end{proof}

Under {\sf CH}, the space $\mathscr K_{\omega_1}^{\mathfrak S}(X)$ has an $\omega$-skeleton when $X$ has an $\omega$-skeleton; and thus all the spaces of $\omega$-universal complemented disposition with $\omega$-skeleton are isometric to $\mathscr K_{\omega_1}^{\mathfrak S}(\R)$. Let us call this
unique space $\mathcal K_{\omega_1}^{\mathfrak S}$ from now on. Since $X$ is $1$-complemented in $K_{\omega_1}^{\mathfrak S}(X)$, under {\sf CH}, $\mathcal K_{\omega_1}^{\mathfrak S}$ contains isometric $1$-complemented copies of every Banach space with $\omega$-skeleton; i.e.,

\begin{proposition}\label{partanswer} Under {\sf CH}, a space  of $\omega$-universal complemented disposition with $\omega$-skeleton contains isometric $1$-complemented copies of all spaces
with $\omega$-skeleton.\end{proposition}

Observe that, even outside {\sf CH}, spaces of $\omega$-universal complemented disposition contain isometric $1$-complemented copies of all separable spaces. It therefore follows from the Johnson-Szankowski theorem \cite{johnszank1} that spaces of $\omega$-universal complemented disposition must have density character at least $\aleph_1$.
Spaces of universal complemented disposition need not be of $\omega$-universal complemented disposition: indeed, ultrapowers of spaces of almost universal disposition are of  universal complemented disposition, although they cannot be of $\omega$-universal complemented disposition since they cannot contain complemented copies of $c_0$.

\section{Open ends}

We leave open a few questions which appeared during the course of this paper.

\begin{enumerate}

\item Is the Kadec space $\mathcal K$ of \cite{kade} of almost universal complemented disposition. Equivalently, is it isometric to $\mathscr  K(\R)$?

\item Does a  space of almost universal complemented disposition contain isometric $1$-complemented copies of all finite-dimensional Banach spaces?

\item Does a separable space of almost universal complemented disposition contain isometric $1$-complemented copies of all separable spaces with $1$-FDD?
Without separability assumption the  answer is no. On the other hand, the spaces $\mathscr K(X)$ are $1$-complementably universal for spaces with $1$-FDD, regardless of whether they have or not skeleton, since $\mathscr K(X)$ contains a $1$-complemented copy of $\mathscr K(\R)$.
\item Do separable spaces of universal complemented disposition exist?
\item Is there a continuum of non-isomorphic spaces of universal complemented disposition? The corresponding question of the existence of many different spaces of universal disposition has been treated, although not completely solved, in \cite{accgm2} and \cite{castsimo}. \item Does a space of universal complemented disposition contain isometric copies of all spaces with $\omega$-skeleton? Observe that a space of $\omega$-universal complemented disposition contains isometric copies of all spaces with $\omega$-skeleton.
    \item Prus shows in \cite[Thm. 2.1]{prus} that there is a reflexive separable space with basis that is complementably universal for all separable super-reflexive spaces with BAP. This suggest the possibility of obtaining other spaces of ``almost universal complemented disposition with respect to certain subclasses of finite dimensional spaces" (see also \cite[Prop. 3.1]{castsimo}). However, we cannot see how the methods in this paper could cover the reflexive case. See also \cite{os}
\end{enumerate}

\section{Appendix: Basic constructions}\label{basic}

We begin with the observation in the Introduction that the spaces of Pe\l cy\'nski, Kadec and Wojtaszczyk are isomorphic. This is consequence of
Pe\l czy\'nski decomposition method.

\begin{lema}\label{mix} Let $\mathcal M$ be a class of quasi-Banach spaces such that for some $0<p\leq \infty$ it is closed under $\ell_p$-sums. There is only one complementably universal member for $\mathcal M$, up to isomorphisms.
\end{lema}
\begin{proof} Let $\mathcal U$ be a complementably universal member. The spaces $\mathcal U$ and $\ell_p(\mathcal U)$
contain complemented copies of each other, and $\ell_p(\ell_p(\mathcal U)) \simeq \ell_p(\mathcal U)$, hence $\mathcal U \simeq \ell_p(\mathcal U)$. In particular,
$\mathcal U \simeq \mathcal U \oplus \mathcal U$. Now, if $A, B$ are two complementably universal members of $\mathcal M$, each of them contains
a complemented copy of the other and both are isomorphic to their squares, so they are isomorphic.\end{proof}

\subsection{Arrows and double arrows}
\adef\label{1menose}
We say that $f : A \to B$ is a $(1 + \varepsilon)$-isometry if it is a linear continuous operator $f : A \to B$ such that for every $x\in A$ verifies $(1 + \varepsilon)^{-1}\|x\|\leq \|f(x)\|\leq (1 + \varepsilon) \|x\|$. We will say that $f$ is a contractive $(1 +\varepsilon)$-isometry if it is a linear continuous operator $f : A \to B$ such that for every $x\in A$ verifies $(1 + \varepsilon)^{-1}\|x\|\leq \|f(x)\|\leq \|x\|$.
\zdef

We define now ``double arrows" $(f, \overline f)$ between Banach spaces.

\adef Given $\alpha, \gamma>1$ and $\beta\geq 0$ a (contractive) $(\alpha,\beta,\gamma)$-arrow is a pair $(f, \overline f)$ of linear continuous operators, $f:A\to B$ and $\overline f:B\to A$ in which $f$ is a (contractive) $\alpha$-isometry, $\|\overline f\| \leq \gamma$ and $\|\overline f f - 1_A\| < \beta$.\zdef

Throughout the paper, $(1,0,1)$-arrows have been called double arrows, and pairs $(f, \overline f)$ which are $(\alpha, \beta, \gamma)$-arrows for suitable $\alpha, \beta, \gamma $ have been called almost double arrows and depicted as $(f, \overline f): A \leftrightarrows B$.
The composition of two almost double arrows is $(f, \overline f)(g, \overline g)= (fg,\overline{gf})$. The operator $\overline f: B\to A$ can be considered as a kind of ``projection". This means that if one has a true
projection $\pi: B\to B$ with range $f(A)$ we will understand that the $\overline f: B \to A$ is $f^{-1}\pi$.
When no confusion arises, given an $(\alpha, \beta, \gamma)$-arrow $(f, \overline f)$ we will simply say that $\overline f$ is a $\beta$-projection along $f$ of norm at most $\gamma$. To measure the commutativity of diagrams we will need a fourth parameter.

\adef Let $(i_1, \overline i_1) : A \leftrightarrows C$, $(i_2, \overline i_2): A \leftrightarrows B$ and $(i_3, \overline i_3): B \leftrightarrows C$ be almost double arrows.
We will say that the diagram they form

$$\xymatrix{
 A \ar[rr]_{i_1} \ar@<-1ex>[rd]_{i_2} && C \ar@<-1ex>[ll]_{\overline {i_1}} \ar[ld]_{\overline {i_3}}\\
&B \ar@<-0.5ex>[lu]_{\overline {i_2}}  \ar@<-1ex>[ru]_{i_3}}
$$


i) $\varepsilon$-commutes if $\|i_3i_2 - i_1\| \leq \varepsilon$ and $\|\overline i_2\overline i_3 - \overline i_1\| \leq
\varepsilon$. ii) Almost commutes if there exists $\varepsilon >0$ such that the diagram
$\varepsilon$-commutes. iii) Commutes if $i_3i_2 = i_1$ and $\overline i_2\overline i_3 = \overline i_1$.\zdef


We present now  a technique that allows one to pass from almost-commutative diagrams with bad projections to commutative diagrams with good projections.

\begin{lema}[Ultraperturbation lemma]\label{equiv} Let $X$ be a Banach space. Given a double arrow $(i, \overline i) : F \leftrightarrows G$ between finite dimensional spaces and a double arrow $(j, \overline j): F \leftrightarrows X$, the following properties are equivalent:
 \begin{enumerate}

\item[i)] For every $\varepsilon >0$ there exists a $(1+\varepsilon, \varepsilon, 1+\varepsilon)$-double arrow $(J,\overline J): G \leftrightarrows X$ making
    the diagram $\varepsilon$-commute.
 \item[ii)] For every $\varepsilon >0$ there exists a $(1+\varepsilon, \varepsilon, 1)$-double arrow $(J, \overline J): G \leftrightarrows X$ making the diagram
     commute.

  \end{enumerate}
 \end{lema}
\begin{proof} It is clear that $ ii)\Rightarrow i)$, so we only need to prove that $i)\Rightarrow ii)$.
Consider a positive sequence $(\varepsilon_n)$ with $\lim \varepsilon_n=0$ and, by i), $(1+\varepsilon_n, \varepsilon_n, 1+\varepsilon_n)$-double
arrows $(J_n, \overline {J_n}): G \leftrightarrows X$ making the diagram $\varepsilon_n$-commute. Take a non-trivial ultrafilter $\U$ on $\N$
and form the operators $[J_n]: G_\U \to X_\U$ and $[\overline {J_n}]: X_\U \to G_\U$. It turns out that $[J_n]$ is an into isometry and
$[\overline {J_n}]$ a norm 1 projection through $[J_n]$, so $([J_n], [\overline {J_n}]): G_\U \leftrightarrows X_\U$ is a double arrow. And if $([i],
[\;\overline i\;]): F_\U \leftrightarrows G_\U$ and $([j], [\;\overline j\;]): F_\U \leftrightarrows X_\U$ are the natural double arrows, the diagram is
commutative since $[J_n][\; i\;]=[\;j\; ]$ and $[\;\overline i\; ][\overline J_n]=[\;\overline j \;]$.\medskip

Since $F$ and $G$ are finite dimensional spaces then $F = F_\U$ and $G
= G_\U$. So $[J_n](G_\U)$ is finite dimensional and we can
choose a $\varepsilon/2$-net $f_1,...,f_N$ in the dual unit ball $B_{[J_n](G)^*}$, which we can assume to be in the dual unit ball of $(X_\U)^*$, such that for every $g\in G$,
$$|<[J_n]g, f_k>|\geq (1+\varepsilon)^{-1}\|[J_n]g\|, \; \mbox{for some }k.$$

Now, observe that the key feature behind the Principle of Local Reflexivity of Lindenstrauss and Rosenthal \cite{plr} is the notion of local complementation identified by Kalton \cite{kaltloc}, as it appears implicitly in \cite{ma} and explicitly in \cite{oja}. We will consider from now on a Banach space $X$  isometrically embedded into its ultrapower $X_\U$ via the map $x\to [x]$. In this form, $X$ is locally complemented in $X_\U$. Thus,
once the functionals $f_1, \dots, f_N$ are set, given $\varepsilon>0$ there is an operator $T_\varepsilon: [J_n](G) \to X$ such that
\begin{enumerate}
\item $\|T_\varepsilon\|\leq 1+\varepsilon$
\item $(T_\varepsilon)_{|[J_n](G)\cup X} = 1_{[J_n](G)\cap X}$
\item $(T_\varepsilon^* f_k)_{|[J_n](G)} = (f_k)_{|[J_n](G)}$
\end{enumerate}

Therefore, the map $T_\varepsilon[J_n]$ is a $1+\varepsilon$-isometry since $\|T_\varepsilon[J_n](g)\| \leq  (1+\varepsilon)\|g\|$ and
\begin{eqnarray*}
\|T_\varepsilon[J_n](g)\| &\geq& |<T_\varepsilon[J_n](g), f_k>|\\
&=& |<[J_n](g), T_\varepsilon^*f_k>|\\ &=& |<[J_n](g), f_k>|\\ &\geq& (1+\varepsilon)^{-1} \|[J_n](g)\|\end{eqnarray*}

On the other hand, the norm 1 projection we need is $[\overline{J_n}]_{|X}$. The couple $(T_\varepsilon[J_n], [\overline{J_n}]_{|X}): G \leftrightarrows X$ is a $(1+\varepsilon, \varepsilon, 1)$-double arrow since, for suitably chosen $f_k$ one has
\begin{eqnarray*}
\|[\overline{J_n}] T_\varepsilon[J_n][g] - [g]\| &=& \|[\overline{J_n}]T_\e[J_n][g] - [\overline{J_n}J_ng]\|\\
&=&\|[\overline{J_n} T_\e[J_n][g]] - [\overline{J_n}J_ng]\|\\
&=&\|[T_\e[J_n][g]] - [J_ng]\|\\
&\leq& \langle [T_\e[J_n][g]] - [J_ng], f_k\rangle + \varepsilon\\
&=& \varepsilon.
\end{eqnarray*}
The diagram commutes since $T_\varepsilon[J_n][i] = T_\e[j] = j$ and $\overline i [\overline{J_n}]_{|X}  = [\overline{i}\overline{J_n}]_{|X} =  \overline{j}_{|X} = \overline{j}.$\end{proof}

Observe that adding ``contractive" to the hypothesis does not improve the results. We conclude this section with a set of elementary estimates that will be useful later.

\begin{lema}$\;$\label{casiequiv}
\begin{enumerate}
\item If $f$ is a (contractive) $(1+\varepsilon)$ isometry and $\tau$ is a (contractive) $(1+\varepsilon')$-isometry then $\tau f $ is a (contractive) $(1+ \varepsilon)(1 + \varepsilon')$-isometry.
\item  If $f$ is a $(1+\varepsilon)$-isometry then $\frac{1}{1+\varepsilon}f$ is a contractive $(1+\varepsilon)^2$-isometry.
\item If $(f, \overline f)$ is an $(\alpha, \beta, \gamma)$-arrow  then $(\frac{1}{\alpha}f, \frac{1}{\gamma}\overline f)$ is a contractive
$\left (\alpha^2, \frac{\beta + \gamma \alpha -1}{\gamma\alpha} , 1 \right)$ - arrow.
\item\ If $(f, \overline f)$ is a (contractive) $(1+\varepsilon, \varepsilon, 1)$-arrow  then $(f, (\overline f\; f)^{-1}\overline f)$ is a (contractive) $(1+\varepsilon, 0, 1 + \frac{\varepsilon}{1-\varepsilon}) - \mathrm{arrow}.$ Moreover, $\|\overline f - (\overline f\; f)^{-1}\overline f\|\leq  \frac{ \varepsilon }{1-\varepsilon}$
\end{enumerate}
\end{lema}
\begin{proof} Probably only assertion (4) requires some explanation. Since $\|1 - \overline f\; f \|\leq \varepsilon <1$ then $1 - (1 - \overline f\; f) = \overline f\; f$ is invertible
and its inverse has norm at most $1 + \varepsilon + \varepsilon^2 + \cdots = \frac{1}{1-\varepsilon} = 1 + \frac{\varepsilon}{1-\varepsilon}$.
Then $(\overline f\; f)^{-1}$ exists and $(\overline f\; f)^{-1}p$ is a true projection along $f$ since $(\overline f\; f)^{-1} \overline f\;  f = 1$. Finally

$$\|\overline f - (\overline f\; f)^{-1}\overline f\| = \|\overline f\; f f^{-1}- (\overline f\; f)^{-1}\overline f\; f f^{-1}\| \leq \|\overline f\; f - 1\|\|f^{-1}\|\leq \frac{ \varepsilon }{1-\varepsilon}.$$

\end{proof}

\subsection{Skeletons}\label{skeleton}
Different approximation notions are essential in the theory of spaces of complemented disposition. A Banach space $X$ is said to have the $\lambda$-approximation property ($\lambda$-BAP in short) if for each finite
dimensional subspace $F\subset X$ and every $\lambda'>\lambda$ there is a finite-rank operator
$T:X\to X$ such that $\|T\|\leq \lambda'$ and $T(f)=f$ for each
$f\in F$. This is not the standard definition, but it is an equivalent formulation (see \cite[Theorem 3.3]{casazza}). The space is said to have the Bounded Approximation
Property (BAP in short) \index{Bounded Approximation
Property, BAP} if it enjoys the \index{$lzap$@$\lambda$-BAP} $\lambda$-BAP for some
$\lambda$. A $\mu$-complemented subspace of a space with the $\lambda$-BAP has the $\lambda\mu$-BAP.  When $X$ is separable, the $\lambda$-BAP is equivalent to the
existence of a sequence $B_n: X \to X$ of linear
finite-dimensional operators with norms $\|B_n\| \leq \lambda$
that is pointwise convergent to the identity. This sequence of operators can be asked to verify $B_mB_n= B_n$ for $m>n$. The sequence is called a
a Finite Dimensional Decomposition (FDD, in short) if, moreover, for every $m, n \in
\N$, $B_nB_m = B_{\min\{m, n\}}$. 
By a well-known
result of Pe\l czy\'nski \cite{pelcbap}, spaces complementably universal for spaces with FDD are also complementably universal
for separable spaces with the BAP.\medskip

An essential part in our arguments and in the classification of spaces of (almost) universal complemented disposition is
played by the notion of \emph{skeleton} which, as we will show next, coincides with that of $1$-Finite Dimensional Decomposition, although the skeleton formulation is
more adapted to the problems treated in this paper:

\adef We say that a Banach space $E$ admits a \emph{skeleton} if there exists a sequence $(E_n)$ of  finite-dimensional subspaces
and of double arrows $(\delta_n, \overline{\delta_n}) : E_n \leftrightarrows E_{n+1}$ so that $E=\overline{\cup_n E_n}$. We will say that $E$  admits
a \emph{$\omega$-skeleton} if there is a continuous chain
$(E_\alpha)_{\alpha<\omega_1}$ of separable subspaces and double arrows $(\delta_\alpha, \overline{\delta_\alpha}) : E_\alpha \leftrightarrows
E_{\alpha+1}$. Here \emph{continuous} means that for every limit ordinal $\beta $ one has $E_\beta = \overline {\cup_{\alpha<\beta} E_\alpha}$.\zdef

In each case we shall say that $(\delta_\alpha, \overline{\delta_\alpha})$ is the family of double arrows defining the ($\omega$) skeleton.
Of course that spaces admitting  a skeleton must be separable and spaces admitting a $\omega$-skeleton must have density character at most
$\aleph_1$.
\begin{lema}\label{garlem}$\;$ \begin{itemize} \item A Banach space has a skeleton if and only if it has a $1$-FDD.
\item A Banach space has an $\omega$-skeleton if and only if it is a $1$-Plichko space with density character at most $\aleph_1$.
\end{itemize}\end{lema}
\begin{proof} Assume that a Banach space $E$ has a skeleton $(\delta_n, \overline{\delta_n}) : E_n \leftrightarrows E_{n+1}$. The spaces $E_k$ are $1$-complemented in $E$ since one can define norm one projections $P_k: E\to E_k$ as follows: if $x\in \cup E_n$ and $x\in E_{n+1}$ then set $P_k(x)=\overline{\delta_k}\dots \overline{\delta_{n-1}}\overline{\delta_n}(x)$ and extend $P_k$ to $E$ by
density. Notice that if $n+1 < k$, then $P_k(x) = x$. These projections verify $\lim P_k(x)=x$. Thus, spaces with skeleton have the $\pi_1$-property \cite[Def.5.1]{casazza}; i.e., there is a net of finite rank norm one projections pointwise convergent to the identity. The $\pi_1$ property in a separable space implies $1$-FDD \cite[Prop.5.4]{casazza}. It is clear that  spaces with $1$-FDD have a skeleton.
The second part can be found in \cite[Section 6]{kubislinear}.
\end{proof}

$\omega$-skeletons will only appear in the final Sections 8 and 9, where we will maintain the name by coherence with the rest of the paper and because statements are shorter this way. The first assertion in Lemma \ref{garlem} appears used in \cite{garbula}. Kubi\'s \cite{kubiskele} and other authors have given more general notions of \emph{projectional skeleton} by considering a partially ordered index space.

\subsection{Push-out constructions}

\subsubsection{The push-out}  Given
operators $i:Y\to A$ and $j:Y\to B$, the associated
push-out diagram is
\begin{equation}\label{po-dia}
\begin{CD}
Y@>i>> A\\
@V j VV @VV j' V\\
B @> i' >> \PO
\end{CD}
\end{equation}
Here, the push-out space $\PO=\PO(i,j)$ is the quotient
of the direct sum $A\oplus_1 B$, the product space endowed with
the sum norm, by the closure of the subspace $\Delta=\{(i
y,-j y): y\in Y\}$. We will call $Q: A\oplus_1 B\to (A\oplus_1 B)/\overline\Delta$,
the natural quotient map. The map $i'$ is given by the
inclusion of $B$ into $A\oplus_1 B$ followed by $Q$, so
that $i'(b)=(0,b)+\overline\Delta$ and, analogously,
$j'(a)=(a,0)+\overline\Delta$.

The diagram (\ref{po-dia}) is commutative:
$j'i=i'j$. Moreover, it is `minimal' in the
sense of having the following universal property: if $j'':A\to
C$ and $i'':B\to C$ are operators such that
$j''i=i''j$, then there is a unique operator
$\gamma:\PO\to C$ such that $i''=\gamma i'$ and
$j''=\gamma j'$. Clearly, $\gamma((a, b) +
\overline{\Delta}) = j''(a)+i''(b)$ and one has
$\|\gamma\|\leq \max \{\|i''\|, \|j''\|\}$. Regarding the
behaviour of the maps in diagram~(\ref{po-dia}) one has (see \cite[Lemma A.19]{accgmLN} for details):

\begin{lema}\label{isom}$\;$
\begin{itemize}
\item[(a)] If $i$ is an isomorphic embedding, then $\Delta$
is closed. \item[(b)] The norm of the operators $i'$ and $j'$ is less than or equal to one.

\item[(c)]If $i$ is an isometric embedding and
$\|j\|\leq 1$ then $i'$ is an isometric embedding.
\item[(d)] If $i$ is an isomorphic embedding then $i'$
is an isomorphic embedding. \item[(e)] If $\|j\|\leq 1$ and
$i$ is an isomorphism then $i'$ is an isomorphism and
$$\|(i')^{-1}\|\leq \max \{1, \|i^{-1}\|\}.$$
\end{itemize}
\end{lema}

\subsubsection{The almost-complemented push-out.} We establish now that the push-out construction can be adapted to cover the case of $\varepsilon$-projections.

\begin{lema}\label{amostdpo} Given almost double arrows $(i,\overline{i}):  A \leftrightarrows B$ and $(j,\overline{j}):  A \leftrightarrows X$ there is a commutative diagram
\begin{equation}\label{thediagram}\xymatrix{
A  \ar[r]\ar@<-1ex>[d]_{(j,\overline{j})}& B\ar@<-1ex>[l]_{(i,\overline{i})} \ar[d]  \\
X \ar@<-1ex>[r]_{(i',\overline{i'})} \ar[u]& \PO \ar[l]\ar@<-1ex>[u]_{(j',\overline{j'})} }\end{equation}
so that if $(i, \overline{i})$ is an $(\alpha, 0 , \gamma)$-arrow and $(j, \overline{j})$ is a  $(u, v, w)$-arrow then
$(i', \overline{i'} )$ is a contractive $(\alpha u, 0, u\gamma)$-arrow and $(j', \overline{j'} )$ is a contractive $(u\alpha, \alpha v \gamma ,\max\{\alpha w, 1+\alpha v\gamma\})$-arrow. Moreover (compare with Lemma \ref{isom} (c) above) if $(i, \overline{i})$ is a $(1, 0, 1)$-arrow, and $(j, \overline{j})$ is a  contractive $(u, v, w)$-arrow
then $(i', \overline{i'} )$ is a $(1, 0, 1 )$-arrow and $(j', \overline{j'} )$ is a contractive  $(u, v,\max\{w, 1+ v\})$-arrow.
\end{lema}
All this can be depicted for mnemonical reasons as
\begin{equation}\label{general}\xymatrix{
\cdot \ar[r]\ar@<-1ex>[d]_{(u,v,w)}& \cdot\ar@<-1ex>[l]_{(\alpha, 0, \gamma)} \ar[d]  \\
\cdot \ar@<-1ex>[r]_{\mathrm{contractive}\; (u\alpha, 0, u\gamma)} \ar[u]& \cdot \ar[l]\ar@<-1ex>[u]_{\mathrm{contractive} \; (u\alpha,\; \alpha v\gamma,\; \max\{w\alpha, 1 + \alpha v\gamma\})}}\end{equation}
and
\begin{equation}\label{101}\xymatrix{
\cdot \ar[r]\ar@<-1ex>[d]_{\mathrm{contractive}\; (u,v,w)}& \cdot\ar@<-1ex>[l]_{(1, 0, 1)} \ar[d]  \\
\cdot \ar@<-1ex>[r]_{(1, 0, 1)} \ar[u]& \cdot \ar[l]\ar@<-1ex>[u]_{\mathrm{contractive} \; (u,\; v,\; \max\{w, 1 + v\})}}\end{equation}

\begin{proof} To obtain $\overline{j'}$ observe that the diagram
\begin{equation}\label{iup}
\begin{CD}
A@>i>> B\\
@VjVV  @VV{1_B + i(\overline{j}j- 1_A)\overline{i}}V\\ X@>>{i\overline{j}}>B
\end{CD}
\end{equation}
is commutative, and thus the universal property of the push-out yields the existence of a unique operator $\overline{j'} : \PO \to B$ such that
\begin{enumerate}
\item[(3.a)] $\overline{j'} \; i' = i \; \overline{j}$;
\item[(3.b)] $\overline{j'} \; j' = 1_B + i(\overline{j}\; j- 1_A)\; \overline{i}$;
\item[(3.c)] $\|\overline{j'} \| \leq \max\{\|i\overline{j}\|, \|1_B + i(\overline{j}\; j- 1_A)\; \overline{i}\|\}.$
\end{enumerate}
Notice that by properties of the push-out construction, $\|i'\|\leq 1$ and $\|j'\|\leq 1$ independently of the norms of $i$ and $j$. To estimate the norm of their inverse maps observe that for every $x\in X$,
\begin{eqnarray*}
\|x\|&\leq &\inf_{a\in A}\{\|x - ja\| +\|ja\|\}\\
&\leq&\inf_{a\in A}\{\|x - ja\| + u\alpha  \|ia\|\}\\
&\leq&u\alpha \|i'(x)\|_{PO};
\end{eqnarray*}
thus $(u\alpha)^{-1}\|x\|\leq\|i'(x)\|\leq \|x\|$. Except when $i$ is an into isometry and $\|j\|\leq 1$, in which case
\begin{eqnarray*}
\|x\|&= &\inf_{a\in A}\{\|x - ja\| +\|ja\|\}\\
&\leq&\inf_{a\in A}\{\|x - ja\| + \|ia\|\}\\
&=&\|i'(x)\|_{PO},
\end{eqnarray*}
and thus $\|x\|= \|i'(x)\|_{PO}$. In the same way, for every  $b\in B$,
\begin{eqnarray*}
\|b\|&\leq &\inf_{a\in A}\{\|b + ia\| +\|i a\|\}\\
&\leq&\inf_{a\in A}\left\{\|b + ia\| + \alpha u\|ja\|\right\},
\end{eqnarray*}
and thus $(u\alpha)^{-1}\|b\| \leq\|j'(b)\|_{PO}\leq \|b\|$. To obtain $\overline{i'}$, since the diagram
\begin{equation}
\begin{CD}
A@>i>> B\\
@VjVV  @VV{j\overline{i}}V\\ X@>>{1_X}>X,
\end{CD}
\end{equation}
is commutative,  the universal property of the push-out yields a unique operator $\overline{i'}: \PO \to X$ such that
\begin{enumerate}
\item[(4.a)] $\overline{i'}\; i' = 1_X$;
\item[(4.b)] $\overline{i'}\; j' = j\;\overline{i}$;
\item[(4.c)]$\|\overline{i'}\|\leq \max \{\|1_X\|, \|j\; \overline{i}\| \}.$
\end{enumerate}

Let us check that the just defined projection $\overline{i'}$ and $\varepsilon$-projection
$\overline{j'}$ make commutative the original diagram (\ref{thediagram}). To this end, it is enough to observe that since diagram
\begin{equation}\label{p} \begin{CD}
A@>>i> B\\
@VVjV  @VV{\overline{j} j \overline{i}}V\\ X@>>{\overline{j}}>A.
\end{CD}
\end{equation}
is commutative, the universal property of the push-out yields a unique operator $\gamma : \PO \to A$ such that
\begin{enumerate}
\item[(5.a)] $\gamma i' = \overline{j}$
\item[(5.b)] $\gamma j' = \overline{j} j \overline{i}$.
\item[(5.c)] $\|\gamma\|\leq \max\{\|\overline{j}\|, \|\overline{j}\;j\;\overline{i}\|\}$.
\end{enumerate}

Since $\overline{j} \; \overline{i'} i' = \overline{j} $ and $\overline{j} \; \overline{i'}\;j' = \overline{j} j \overline{i}$ (by (4.b)), the uniqueness (see (5.a) and (5.b)) yields
$\gamma = \overline{j} \; \overline{i'}$. On the other hand, also $\gamma = \overline{i} \; \overline{j'}$ since $\overline{i}\; \overline{j'}\; i' = \overline{j}$ (by (3.a)) and $\overline{i} \; \overline{j'} j' =  \overline{i}(1_B + i(\overline{j}j- 1_A)\overline{i}) = \overline{i} + (\overline{j} \;j - 1_A) \overline{i} = \overline{j}\;  j \;\overline{i}$. \end{proof}


Modifying the proof above in an obvious way we obtain the result of Kubis \cite[Section 5]{kubis} (see also \cite[Lemma 4.1]{garbula} and the comments before the lemma) that in a push-out
diagram
$$
\begin{CD}
A@>>i> B\\
@VjVV @VVj'V\\
C@>>i'>\PO
\end{CD}$$
in which both $i,j$ have complemented ranges via projections $p,q$ then also $i',j'$ have complemented ranges via projections
$p',q'$ yielding a diagram
$$
\begin{CD}
A@>{\stackrel{p}\longleftarrow}>i> B\\
@VjV{\uparrow q}V @Vj'V{\uparrow q'}V\\
C@>{\stackrel{p'}\longleftarrow}>i'>\PO
\end{CD}$$
commutative in both directions i.e., $pq'=qp'$ and, moreover, such that $jp=p'j'$ and $iq=q'i'$.
One has to proceed just as the proof of Lemma \ref{amostdpo} but, in diagram (\ref{iup}), take $1_B$ instead of $i(qj -1_A)p + 1_B$ and, in diagram (\ref{p}), take $p$ instead of $qjp$.

\subsubsection{The complementation feature of multiple push-out} Let us check now that almost complementation is preserved in
almost complemented push-out with several factors:

\begin{lema}\label{POprojection} Let $(i_1, \overline{i_1}): A_1\to B_1$ and $(i_2, \overline{i_2}): A_2\to B_2$ be $(1,0,1)$-arrows. Let $(j_1, \overline{j_1}): A_1\to X$ be a $(u, v, w)$-arrow and let $j_2: A_2\to X$ be an operator. Consider the push-out diagram
\begin{equation}\label{POcop}
\begin{CD}
A_1 \oplus_1 A_2 @>{i_1\oplus \;i_2}>> B_1\oplus_1 B_2\\
@V{j_1+ j_2} VV @VV{J}V\\
X@>>{(i_1\oplus \;i_2)'}>\PO.
\end{CD}
\end{equation}

The restriction $J_{|B_1}$ admits an arrow $\overline{J_{|B_1}}: \PO\to B_1$ so that $(J_{|B_1}, \overline{J_{|B_1}})$ is a contractive $\left(u, v, \max\{w, 1+v\}u\|j_2\|\right)$-arrow.
In particular, if $(j_1, \overline{j_1})$ is a contractive $(u, v, w)$ arrow then $(J_{|B_1}, \overline{J_{|B_1}})$ is a contractive $\left(u, v, \max\{w, 1+v\}\|j_2\|\right)$-arrow.\medskip

Moreover, $\overline {i_1} \overline{J_{|B_1}} = \overline{j_1} \overline{(i_1\oplus i_2)'}$.

\end{lema}

\begin{proof} Perform first the almost-complemented push-out as in the diagram (\ref{thediagram}) in Lemma \ref{amostdpo} to get
\begin{equation}\label{j1prima}\xymatrix{
A_1  \ar[r]\ar@<-1ex>[d]_{(j_1,\overline{j_1})}& B_1\ar@<-1ex>[l]_{(i_1,\overline{i_1})} \ar[d]  \\
X \ar@<-1ex>[r]_{(i_1',\overline{i_1'})} \ar[u]& P_1 \ar[l]\ar@<-1ex>[u]_{(j_1',\overline{j_1'})} }\end{equation}
in which $(i_1',\overline{i_1'})$ is a contractive $(u, 0, u)$-arrow and $(j_1',\overline{j_1'})$ is a contractive $(u, v, \max\{w, 1+v\})$-arrow. Now make the push-out of the arrows $i_2$ and $i'_1j_2$
\begin{equation}\label{i2prima}\xymatrix{
A_2  \ar[r]\ar@<-1ex>[d]_{i_1' j_2}& B_2\ar@<-1ex>[l]_{(i_2,\overline{i_2})} \ar[d]^{(i_1'j_2)'}  \\
P_1 \ar@<-1ex>[r]_{(i_2',\overline{i_2'})}& P_2 \ar[l]}\end{equation}The map $\overline{i'_2}$ is obtained according to diagram (10) in the proof of Lemma \ref{amostdpo}, in such a way that $\overline{i'_2}i'_2 = 1_{P_1}$ and $\|\overline{i'_2}\|\leq max\{1, \|i'_1j_2\overline{i_2}\|\}$.
On the other hand, since the following square is commutative
\begin{equation}\label{POcop}
\begin{CD}
A_1 \oplus_1 A_2 @>{i_1\oplus \;i_2}>> B_1\oplus_1 B_2\\
@V{j_1+j_2} VV @VV{i_2'j_1' + (i_1'j_2)'}V\\
X@>>{i_2' i_1'}>P_2.
\end{CD}
\end{equation}
there must be a unique operator arrow $\tau: \PO\to P_2$ such that
\begin{enumerate}
\item $\tau (i_1 \oplus i_2)' =  i_2' i_1'$
\item $\tau J = i_2'j_1' + (i_1'j_2)'$
\item $\|\tau\|\leq \max \{ \|i_2' i_1'\|,  \|i_2'j_1' + (i_1'j_2)'\|\}\leq  \max \{\|i_2' i_1'\|, \|i_2'j_1' + (i_1'j_2)'\|\}=1$.
\end{enumerate}
The almost projection is going to be $\overline{J_{|B_1}} = \overline{j_1'}\;\overline{i_2'} \tau: \PO \to B_1$, where $\overline{j_1'}$ has been obtained in diagram (\ref{j1prima}) while $\overline{i_2'}$ has been obtained in diagram (\ref{i2prima}). To check this observe that
\begin{eqnarray*} \|\overline{j_1'}\;\;\overline{i_2'}\; \tau \;J_{|B_1} - 1_{B_1}\| &=&  \|\overline{j_1'}\; \;\overline{i_2'}\; ( i_2'j_1' + (i_1'j_2)')_{|B_1}- 1_{B_1}\| \\
& =& \|\overline{j_1'}\; \;\overline{i_2'}\;  i_2'j_1'- 1_{B_1}\| \\
&=& \|\overline{j_1'}\;  \; j_1'- 1_{B_1}\| \\
&\leq& v. \end{eqnarray*}

Since $\|\overline{j_1'}\;\overline{i_2'} \tau \| \leq \max\{w, 1+v\}u\|j_2\|$ it turns out that $(J_{|B_1}, \overline{J_{|B_1}} )$ is a contractive $\left(u, v, \max\{w, 1+v\}u\|j_2\|\right)$-arrow. If $\|j_1\|\leq 1$ then  $(J_{|B_1}, \overline{J_{|B_1}})$ is a contractive $\left(u, v, \max\{w, 1+v\}\|j_2\|\right)$-arrow. Finally, according again to diagram (10) in the proof of Lemma \ref{amostdpo}, there exists and operator $\overline{(i_1\oplus i_2)'}: PO\to X$ such that $\overline{(i_1\oplus i_2)'}(i_1\oplus i_2)'=1_X$, and the ``moreover" part is clear.\end{proof}

\subsubsection{The almost-push-out}\label{apo}

Garbulinska introduces in \cite[Lemma 3.1]{garbula} a useful correction lemma. Let us show that it can be
understood as an ``almost" push-out construction, which moreover admits an extension to cover the case of almost double arrows.
\begin{lema}[Correction lemma]\label{garblem}\label{compgarbu}$\;$
\begin{itemize}
\item Given a  contractive $1+\varepsilon$-isometry $f: X \to Y$ between Banach spaces, there exists a space $E(f, X, Y)$ and isometries
$i_f : X \to E(f, X,Y)$, $j_f: Y \to
E(f, X, Y)$ such that $\|j_ff -i_f\|\leq \varepsilon$ with the following universal property: for any couple of arrows $k:X\to V$ and
$l:Y\to V$ such that $\|lf - k\|\leq \varepsilon$ there exists a unique arrow $\gamma:E(f, X, Y) \to V$ such that $\gamma i_f = k$ and
$\gamma j_f = l$.

\item Given a contractive $(1 +\varepsilon, \e, 1)$-arrow $(f, \overline f):X\leftrightarrows  Y$ there exist a space $E= E(f, X, Y)$ and double arrows $(i, \overline i) : X \leftrightarrows E$, $(j, \overline j): Y \leftrightarrows E$ making the diagram
$$\xymatrix{
 X \ar@<-1ex>[dd]_{f} \ar@<0.5ex>[rd]_{i} \\
   & E(f, X, Y)\ar@<-2ex>[lu]_{\overline i} \ar@<2ex>[ld]^{\overline j}\\
 Y\ar[uu]_{\overline f} \ar@<-0.5ex>[ru]^{j}}$$
$\varepsilon$-commutative and verifying
also $\overline i j = \overline f$ and $\overline j i = f$.
\end{itemize}
\end{lema}
\begin{proof} Let us first see that there exists a push-out diagram which partially corrects the almost-isometry $f$. To this
purpose, consider the isometric (for $\varepsilon<1$) embedding operator $\delta_\varepsilon : X \to X \oplus_\infty X$,
$\delta_\varepsilon(x) = (x, \varepsilon x)$ and make the push-out square
\begin{equation*} \begin{CD}
X@>{\delta_\varepsilon}>> X\oplus_\infty X\\
@VfVV  @VVf'V\\ Y@>>\delta'>\PO.
\end{CD}
\end{equation*}
By the general properties of the push-out, $f'$ is a $(1+\varepsilon)$-isometry and $\delta'$ is an into isometry.
Recall that $\PO$ is the quotient of $(X\oplus_\infty X)\oplus_1 Y$ via the natural quotient map $Q: (X\oplus_\infty X)\oplus_1 Y \to \PO$ with kernel $X$ that defines the push-out.
We form a subspace of
$\PO$ where $X$ and $Y$ embed isometrically at the cost of loosing commutativity by taking
$$E(f,X,Y) = Q\left( X \oplus_\infty 0 \oplus_1 Y \right)$$
and define the map $s: X\oplus_\infty X\to E(f,X,Y)$ by $s(x,z) = f'(x,0)$. The map $\delta'$ is already well defined as a map
$Y\to E(f,X,Y)$. The resulting square
\begin{equation*}\begin{CD}
X@>{\delta_\varepsilon}>> X\oplus_\infty X\\
@VfVV  @VVsV\\ Y@>>\delta'>E(f,X,Y)
\end{CD}
\end{equation*}
is $\varepsilon$-commutative:
$$\|\delta'f(x) - s \delta_\varepsilon(x)\|= \|\overline{(x, 0, -f(x))}\|\leq \inf_{\omega\in X}\|(x
-\omega, \varepsilon \omega, f(\omega)-f(x)\|\leq \varepsilon \|x\|.$$

Moreover $s\delta_\varepsilon: X\to E(f,X,Y)$ is an into isometry:

\begin{eqnarray*}\|x\| &\leq& \|x\| - \|\omega\| + \|\omega\|-\|f(\omega)\| + \|f(\omega)\|, \; \forall \omega\in
X\\
&\leq&\|x - \omega\| + \varepsilon \|\omega\|+ \|f(\omega)\|\\
&\leq& \|\overline{(x,0,0)}\|_{\PO} = \|s\delta_\varepsilon(x)\|_{E}.\end{eqnarray*}

We must therefore set: $i_f = s\delta_\varepsilon$ and $j_f= \delta'$. \medskip

We prove now the universal property mentioned above: let $k: X \to V$ and $l: Y \to V$ be operators such that $\|lf - k\|\leq
\varepsilon$. The map $t: X\oplus X \to V$ defined by $t(x, z) = k(x) +
\varepsilon^{-1}(lf-k)(z)$ verifies $t\delta_\varepsilon = lf$. By the universal
property of the push-out there exists a unique arrow $\gamma : \PO \to V$ such that $\gamma f' = t$ and $\gamma \delta' = l$.
And the recontractiveion of $\gamma$ to $E(f,X,Y)$ yields  $\gamma i_f(x) = \gamma f' (x,0) = t(x,0) =
k(x)$; while $\gamma j_f=\gamma \delta'= l$.\medskip

The complemented version of the Correction lemma  will follow from the universal
property of the ``almost push-out" applied first to the arrows
$1_X$ and $\overline f: Y \to X$, so we get $\overline i: E \to X$ such that $\overline i j = \overline f
$ and $\overline i i = 1_X$; and then to $f$ and $1_Y$, obtaining $\overline j: E \to Y$ such that $\overline j j = 1_Y$ and $\overline j i = f$. In addition,
$\|\overline f\|\leq \|\overline i\|$ and $\|f\|\leq \|\overline j\|$.\medskip

Now, when one has a push-out diagram
\begin{equation*}
\begin{CD}
\bullet @>\alpha>> \bullet\\
@V \beta VV @VV \beta' V\\
\bullet @> \alpha' >> \PO
\end{CD}
\end{equation*}

and two arrows $\gamma: \PO \to Z$ and $h:\PO\to Z$ so that
$\gamma \beta' = h\beta'$ and $\|\gamma \alpha'- h\alpha'\|\leq \varepsilon$ then $\|\gamma - h\|\leq \varepsilon$: indeed, for
given $(c,b)+\Delta\in \PO$ with $\|(c,b)+\Delta\|\leq 1$ pick a representative $(c_1, b_1)+\Delta$ so that $\|c_1\|\leq 1$.
Since $(c_1, b_1)+\Delta = \alpha'(c_1) + \beta'(b_1)$ one has
$$\|(\gamma - h)((c,b)+\Delta)\| = \|(\gamma - h)((c_1, b_1)+\Delta)\| = \|\gamma \alpha'(c_1) - h\alpha'(c_1) \|\leq
\varepsilon.$$

Thus, since $\overline f \;\overline j \;j = \overline f$ and $\overline f\; \overline j\; i=  \overline f f$ and $\overline i \;j = \overline f$ and $\|\overline i\; i - \overline f \;f\|\leq \varepsilon$, it turns out that $\|\overline f \;\overline j - \overline i\|\leq
\varepsilon$.
\end{proof}

The condition $\|\overline f \; \overline j - \overline i\|\leq \varepsilon$ that we have obtained does not appear in either \cite{garbula} or \cite{cgk}, where the authors only consider the almost commutativity condition
$\|jf -i\|\leq \varepsilon$ for embeddings. Observe that the almost commutativity for embeddings and
projections implies $\|\overline i \;j - \overline f\|\leq \varepsilon$ and $\|\overline j \;i - f\|\leq \varepsilon$ (but there is no equality).
\newpage
\subsection{Perturbation of projections}

\begin{lema}\label{close} Let $A$ be an $n$-dimensional subspace of $E$ which is complemented by some projection $p$ of norm $C$. Let
$\delta= \mathrm{dist}(A, \ell_1^n)$. Let $\{a_1,\dots, a_n\}$ be a basis for $A$ so that $ \delta^{-1} \sum |\lambda_i|\leq \|\sum
\lambda_ia_i\|\leq \sum|\lambda_i|$. Given $0<\varepsilon<1/3$, if $\|x_i-a_i\|\leq \frac{\varepsilon}{\delta C}$ then the map $\tau a_i=x_i$ is a $(1+\varepsilon)$-isometry and the space $X=[x_1,\dots,x_n]$ is:
\begin{enumerate}
\item complemented via some projection $p'$ of norm at most $C\frac{1-\varepsilon^2}{1-3\varepsilon}$ for which
$$\|p' - \tau p\|\leq \varepsilon\frac{(1+\varepsilon)^2}{1 - \varepsilon}C.$$
\item In particular, $\frac{1}{1+\e}\tau$ is a contractive $(1+\e)^2$-isometry with projection $(1+\e)p'$ having norm at most $C\frac{(1+\varepsilon)(1-\varepsilon^2)}{1-3\varepsilon}$
and so that
$$\|(1+\e)p' - \frac{1}{1+\e}\tau p\|\leq C3\e.$$
\end{enumerate}
\end{lema}
\begin{proof} The operator $\tau: A\to X$ that sends $\tau(a_i)=x_i$ is a $(1+\varepsilon)$-isometry. And if $p:E\to A$ is a norm-one
projection, on every $x=\sum \lambda_ix_i \in X$ one has

$$\|\tau px- x\| = \|\tau p(\sum \lambda_ix_i) - \sum \lambda_i \tau pa_i\|\leq (1+\varepsilon)\|\sum \lambda_i(x_i-a_i)\| \leq
\varepsilon\frac{1+\varepsilon}{1-\varepsilon}\|x\|.$$
The estimates now are as in Lemma \ref{casiequiv} (4). We call $\mu =  \varepsilon\frac{(1+\varepsilon)}{1-\varepsilon}$. Since $(1_E-\tau p)_{|X}$ has norm $\mu<1$ for $\varepsilon < 1/3$, then $\tau p_{|X} = 1_E - (1_E-\tau p_{|X})$
is invertible and its inverse has norm at most $1 + \mu + \mu^2 + \cdots = \frac{1}{1-\mu}$. So, $(\tau p)_{|X} $ is an isomorphism
and $\|{(\tau p)_{|X}}^{-1}\|\leq \frac{1}{1-\mu}$. It turns out that $p' = {(\tau p)_{|X}}^{-1} \tau p$ is a projection onto $X$ since $$p'^2 = {(\tau p)_{|X}}^{-1}\tau p {(\tau p)_{|X}}^{-1}\tau p =
{(\tau p)_{|X}}^{-1}(\tau p)_{|X} {(\tau p)_{|X}}^{-1}\tau p = {(\tau p)_{|X}}^{-1}\tau p,$$ with norm at most $\frac{1 + \varepsilon}{1
-\mu}.$ Moreover,
\begin{eqnarray*} \|p' - \tau p\| &=& \|(\tau p)_{|X}^{-1} \tau p- \tau p\| \\
&\leq &\|(\tau p)_{|X}^{-1} - 1_X \|\|\tau p\| \\
&\leq &\|(\tau p)_{|X}^{-1}\|\|1_X - (\tau p)_{|X} \|\|\tau\| \\
&\leq & \frac{1}{1 - \mu} \mu (1+\varepsilon)\\
&\leq & \frac{\e(1+\e)^2}{1 - 3\e}\\
&\leq & \frac{4\e}{1 - 3\e}.
\end{eqnarray*}\end{proof}

The ``in particular" estimate easily follows:

\begin{eqnarray*}\|(1+\e)p' - \frac{1}{1+\e}\tau p\|&=& (1+\e) \|p' - \frac{1}{(1+\e)^2}\tau p\|\\
&\leq&(1+\e) \left( \|p' - \tau p\| + \left( 1- \frac{1}{(1+\e)^2}\right)\| \tau p\|\right)\\
&\leq&(1+\e) \left( \|p' - \tau p\| + \left( 1- \frac{1}{(1+\e)^2}\right)\| \tau p\|\right)\\
&\leq& (1+\e) \left( \frac{\e(1+\e)^2}{1 - 3\e}C + \frac{(1+\e)^2 -1}{(1+\e)^2}(1+\e)C\right)\\
&\leq& C \left( \frac{\e(1+\e)^3}{1 - 3\e} +  2\e +\e^2 \right)\\
&\leq& C ( 3\e - \e^2 )\\
&\leq& C 3\e.
\end{eqnarray*}

\subsection{Countable dense sets of double arrows between finite-dimensional spaces}\label{distance} To produce a separable space as output a
basic ingredient is to have a countable set of double arrows between finite dimensional spaces that is ``dense". To
this end, consider for fixed $n\leq k$ the set of double arrows
$$\mathcal U_{n,k} = \{(f, \overline f): A\leftrightarrows B\;\quad \dim A =n;\quad \dim B = k\}$$
in which elements are identified as: $(f, \overline f) \sim (g, \overline g)$ when there are surjective isometries $a : A\to A'$ and $b : B\to B'$ such that
$bf =ga $ and $a\overline f= \overline g b$. We call $\mathcal U(n,k)$ the quotient space endowed with the metric induced by
$$d((f, \overline f), (g, \overline g)) = \inf \{\log(1+\varepsilon)>0: \exists a,b\;(1+\varepsilon)-\mathrm{onto\; isometries}:  bf =ga  \;\mathrm{and}\; a\overline f= \overline g b\}.$$

One has:

\begin{lema}\label{metricdouble} The space $\mathcal U(m, l)$ is a compact metric space for all $m, l$.
\end{lema}\begin{proof}
Let $(A_k,B_k, f_k, \overline {f_k})$ be a sequence. In the Banach-Mazur distance --for spaces-- and the operator norm --for operators--
there is a subsequence (no need to relabel) so that $\lim A_k = A$, $\lim B_k = B$, $\lim f_k = f$ and
$\lim \overline {f_k} = \overline{f}$. There is no loss of generality assuming that the almost isometries that yield the Banach-Mazur distance are the
identity. Which in particular means that if one fixes a basis in each $A_k$ and
$e_j^k$ is the j-th element in $A_k$ then  $e_j^k\to a_j$, the elements
$a_j$ form a basis for $A$ and $f_k(a_j) \to b_j$ form a basis for $f(A)$ in $B$, which we complete with as many $b_i's$ as
necessary. Let $\U$ be  a free ultrafilter on $\N$. One has $A=[A_1, A_2, \dots, A_n,... ]_\U$ and $B=[B_1, B_2, \dots,
B_n,...]_\U$. The map $f=[f_k]$ is thus an isometry between them and  $\overline{f}=[\overline {f_k}]$ a 1-projection. Moreover, given a finite dimensional space $F$ one has $F=[F]_\U$ and thus one can identify $A_k$ with the its ultrapower $[A_k]_\U$ and  $B_k$ with $[B_k]_\U$. In this way, the formal identity $1_k: \; \; A_k = [A_k]_\U \to [A_n]_\U = A$ is a
$1+\mathrm{dist}(A_k,A)$-isometry. To check that $(f, \overline{f}): A \leftrightarrows B$ is the limit of $(f_k, \overline{f_k})$
we set $1_k$ on the left and do as follows on the right:
given $k$, we call $t_k: B_k \to B$ the $1+\e$-isometry that fixes all $b_i$ while sending $f_k(a_j)$ to $b_j$ (of course that
$\e$ depends on $k$, but goes to $0$ when $k$ goes to infinity). Form $[t_k]$ and observe that the diagram

\begin{equation*}
\begin{CD}
A_k@>{\stackrel{\overline {f_k}}\leftarrow}>f_k> B_k\\
@V1_k VV @VV[t_k] V\\
[A_n]_\U@>{\stackrel{[\overline {f_k}]}\leftarrow}> [f_n] > [B_n]_\U
\end{CD}
\end{equation*}
is commutative in both directions.\end{proof}

Now, observe that there is no loss of generality in assuming that the two $(1+\varepsilon)$-isometries $a, b$ in the definition of the distance $d(\cdot, \cdot)$ at the beginning of section \ref{distance} are
contractive $(1+\varepsilon)$-isometries: indeed, given $a, b$ so that $bf =ga $ and $a\overline f= \overline g b$ one can set
$a'=\frac{1}{1+\varepsilon}a$ and $b'=\frac{1}{1+\varepsilon}b$, who still satisfy  $b'f =ga' $ and  $a'\overline f= \overline g b'$. For the same reason, one can
also make $a^{-1}, b^{-1}$ contractive  $(1+\varepsilon)$-isometries. Thus, since metrizable compacta are separable we get:

\begin{lema}\label{denseness} There is a countable set $\mathfrak U$ of double arrows between finite dimensional spaces with the
following property:
given a double arrow $(w, \overline w): A\leftrightarrows B$ between finite dimensional spaces and $\varepsilon>0$, there is
$(u, \overline u): A_u \leftrightarrows B_u$ in $\mathfrak U$,  and surjective contractive $(1+ \e)$-isometries
$a: A_u\to A$ and $b: B_u\to B$ making the square
\begin{equation}\label{wu}
\begin{CD}
A_u@>u>> B_u\\
@Va VV @VVb V\\
A@> w >> B
\end{CD}
\end{equation}
commutative both directions; i.e., $wa = bu$ and $a\overline u  = \overline w b$.
\end{lema}


\subsection{Distances between double arrows and the role of dual separable spaces}

Almost double arrows $A\leftrightarrows B$ form a subset of $\mathfrak L(A, B) \oplus \mathfrak L(B, A)$, and thus the distance between two almost double arrows $(j, \overline j): A\leftrightarrows B$ and $(i, \overline i): A\leftrightarrows B$ is defined as $\max \{ \|j-i\|, \|\overline j - \overline i\|\}$. The following lemma is here to justify the additional hypothesis in Theorem \ref{KX}.
    \begin{lema} Let $F$ be a finite dimensional Banach space. There is a countable set of double arrows $F\to X$ which is dense in the set of all double arrows $F\to X$ if and only if $X^*$ is
    separable.
\end{lema}
\begin{proof} (Necessity) Set $F=\R$, without loss of generality. Every double arrow $(f,p): \R\leftrightarrows  X$ is an isometric embedding $f$ and a
$1$-projection onto $f(\R)$. Or, which is the same, a norm one element $u\in X$ and a norm one functional $\phi\in X^*$ so that
$\phi(u)=1$. The projection is $p(x)=\phi(x)u$. Assume there is a countable set of $(f,p)$ so that for every $(g,q)$ there is
one of them for which $\|f-g\|+\|p-q\|\leq \varepsilon$. Let $\psi$ be a norm one element of $X^*$. Find norm one $v\in X$ for
which $\psi(v)=1-\varepsilon$ and then form the isometric embedding $g(1)=v$ with projection
$q(x)=(1-\varepsilon)^{-1}\psi(x)v$. Find one of those countable elements $(f,p)$ close to $(g,q)$. If $p(x)=\phi(x)u$ with
$f(1)=u$ then for $\|x\|=1$ one has

\begin{eqnarray*} |\phi(x) &-& (1-\varepsilon)^{-1}\psi(x)| = \|\phi(x)u - (1-\varepsilon)^{-1}\psi(x)u\|\\
&\leq& \|\phi(x)u - (1-\varepsilon)^{-1}\psi(x)v\| + \|(1-\varepsilon)^{-1}\psi(x)v - (1-\varepsilon)^{-1}\psi(x)u\| \\
&\leq& \|p - q\| + (1-\varepsilon)^{-1}\|v - u\| \\
&=& \|p - q\| + (1-\varepsilon)^{-1}\|f(1) - g(1)\| \\
&\leq& \|p - q\| + (1-\varepsilon)^{-1}\|f - g\| \\
&\leq&2\varepsilon + \frac{2\varepsilon}{1-\varepsilon}.
\end{eqnarray*}

(Sufficiency) The set of double arrows so it is separable when both  $\mathfrak L(F,X)$ and $\mathfrak
L(X,F)$ are separable; that is, when $X^*$ is separable.\end{proof}

\end{document}